\documentclass[10pt]{article}

\usepackage[T1]{fontenc}
\usepackage{lmodern}
\usepackage[utf8]{inputenc}
\usepackage{microtype}
\usepackage{framed}
\usepackage{listings}
\usepackage{vmargin}
\usepackage{setspace}
\usepackage{mathrsfs, mathenv}
\usepackage{amsmath, amsthm, amssymb, amsfonts, amscd}
\usepackage{graphicx}
\usepackage{epstopdf}
\usepackage[svgnames]{xcolor}
\usepackage{hyperref}
\hypersetup{citecolor=blue, colorlinks=true, linkcolor=black}
\usepackage{subcaption}
\usepackage{bbm}
\usepackage{cite}
\usepackage{verbatim}
\usepackage{pgfplots}
\usepackage{tikz}
\usepackage{etoolbox}
\usepackage{color}
\usepackage{lipsum}
\usepackage{ifthen}
\usepackage[linesnumbered, ruled, vlined]{algorithm2e}
\usepackage{soul}
\usepackage{multicol}

\theoremstyle{plain}
\newtheorem{theorem}{Theorem}

\newtheorem{lemma}{Lemma}
\newtheorem{proposition}{Proposition}

\theoremstyle{definition}
\newtheorem{definition}{Definition}

\theoremstyle{remark}
\newtheorem{remark}{Remark}
\newtheorem{assumption}{Assumption}

\ifpdf
  \DeclareGraphicsExtensions{.eps,.pdf,.png,.jpg}
\else
  \DeclareGraphicsExtensions{.eps}
\fi

\usepackage{mathtools}
\mathtoolsset{showonlyrefs}

\usepackage{booktabs}

\usepackage{pgfplots}
\usepackage{tikz}
\usetikzlibrary{patterns,arrows,decorations.pathmorphing,backgrounds,positioning,fit,matrix}
\usepackage[labelfont=bf]{caption}
\usepackage{here}
\usepackage[font=normal]{subcaption}

\usepackage{enumitem}
\setlist[itemize]{leftmargin=.5in}
\setlist[enumerate]{leftmargin=.5in,topsep=3pt,itemsep=3pt,label=(\roman*)}

\newcommand{\TheTitle}{Ensemble Kalman filter for multiscale inverse problems} 
\newcommand{\TheAuthors}{A. Abdulle, G. Garegnani}
\title{{\TheTitle}}
\newcommand*\samethanks[1][\value{footnote}]{\footnotemark[#1]}
\author{Assyr Abdulle\thanks{Institute of Mathematics, \'Ecole Polytechnique F\'ed\'erale de Lausanne}
	\and
	Giacomo Garegnani\samethanks
	\and 
	Andrea Zanoni\samethanks
	}
\date{}

\newcommand{\toweak}{\rightharpoonup}

\usepackage{amsopn}

\newcommand{\abs}[1]{\left\lvert#1\right\rvert}
\newcommand{\norm}[1]{\left\|#1\right\|}
\renewcommand{\phi}{\varphi}
\renewcommand{\theta}{\vartheta}

\newcommand{\N}{\mathbb{N}}
\newcommand{\R}{\mathbb{R}}

\newcommand{\epl}{\varepsilon}

\newcommand{\defeq}{\coloneqq}
\newcommand{\eqdef}{\eqqcolon}

\newcommand{\E}{\operatorname{\mathbb{E}}}

\definecolor{shade}{RGB}{100, 100, 100}
\definecolor{bordeaux}{RGB}{128, 0, 50}

\usepackage[usestackEOL]{stackengine}

\definecolor{leg1}{RGB}{0,114,189}
\definecolor{leg2}{RGB}{217,83,25}
\definecolor{leg3}{RGB}{237,177,32}
\definecolor{leg4}{RGB}{126,47,142}
\definecolor{leg5}{RGB}{119,172,48}

\definecolor{leg21}{RGB}{62,38,169}
\definecolor{leg22}{RGB}{46,135,247}
\definecolor{leg23}{RGB}{55,200,151}
\definecolor{leg24}{RGB}{254,195,56}

\ifpdf
\hypersetup{
	pdftitle={\TheTitle},
	pdfauthor={\TheAuthors}
}
\fi

\begin{document}
\maketitle

\begin{abstract} We present a novel algorithm based on the ensemble Kalman filter to solve inverse problems involving multiscale elliptic partial differential equations. Our method is based on numerical homogenization and finite element discretization and allows to recover a highly oscillatory tensor from measurements of the multiscale solution in a computationally inexpensive manner. The properties of the approximate solution are analysed with respect to the multiscale and discretization parameters, and a convergence result is shown to hold. A reinterpretation of the solution from a Bayesian perspective is provided, and convergence of the approximate conditional posterior distribution is proved with respect to the Wasserstein distance. A numerical experiment validates our methodology, with a particular emphasis on modelling error and computational cost.
\end{abstract}



\section{Introduction}

In this work we consider the application of techniques derived from the Kalman filter to inverse problems involving multiscale phenomena which can be modelled by means of partial differential equations (PDEs). Inverse problems arise in many fields, such as seismography, meteorology and tomography, all physical domains with a multiscale nature. Our reference mathematical model is given by multiscale elliptic PDEs of the form
\begin{equation}
\left\{
\begin{alignedat}{2}
- \nabla \cdot ( A^{\varepsilon}_u \nabla p^{\varepsilon} ) &= f, \quad && \text{ in } \Omega, \\
p^{\varepsilon} &= 0, \quad && \text{ on } \partial \Omega,
\end{alignedat}
\right.
\end{equation}
where $\Omega \subset \R^d$ is the physical domain, $A_u^\epl$ is a tensor oscillating with an amplitude described by the parameter $\epl$ and $u$ is a possibly infinite-dimensional unknown which parametrizes the tensor $A^\epl_u$. We are then interested in the solution of inverse problems involving the retrieval of the parameter $u$ given noisy observations derived from the solution $p^\epl$.

Multiscale inverse problems of this form have been recently introduced in \cite{NPS12} and analysed extensively in \cite{AbD19, AbD20}. In particular, in \cite{AbD19} Abdulle and Di Blasio build a coarse-graining approach to solve the inverse problem regularized with a Tikhonov technique. The main idea is replacing the computationally expensive solution of the highly-oscillating multiscale problem with an homogenized surrogate, which eliminates the fast variables and is therefore cheaper. In particular, the theory of homogenization guarantees under certain assumptions, which will be specified throughout this work, that there exists a PDE of the form
\begin{equation}
\left\{
\begin{alignedat}{2}
- \nabla \cdot ( A^0_u \nabla p^0 ) &= f, \quad && \text{ in } \Omega, \\
p^0 &= 0, \quad && \text{ on } \partial \Omega,
\end{alignedat}
\right.
\end{equation}
such that the solution $p^0$ is the weak limit of the functions $p^\epl$ in the vanishing limit for $\epl$, and such that $A_u^0$ is independent of $\epl$. In \cite{AbD19}, the authors showed that employing this homogenized model to the multiscale inverse problem guarantees a good approximation to its solution if a Tikhonov regularization is employed. This framework has been successively enlarged by the same authors to the Bayesian case in \cite{AbD20}, where the analysis involves posterior distributions arising from both the multiscale and the homogenized model. In the same work, a technique for estimating the modelling error which was developed in \cite{CES14, CDS18} is successfully applied to multiscale inverse problems to account for the homogenization and discretization errors.

The ensemble Kalman filter (EnKF), first introduced in \cite{Eve94}, is an algorithm which is widely employed in the engineering community for the estimation of the state of partially-observed dynamical systems whose dynamics are governed by a nonlinear agent. In particular, Kalman filters have long been used successfully in meteorology, oceanography and automation applications. In \cite{ILS13}, Iglesias et al. propose the application of the EnKF method to obtain a point-wise solution to inverse problems involving PDEs, and an extension of their analysis giving a Bayesian interpretation of the filtering solution is presented in \cite{ScS17}.

In this work, we present a combination of the well-established techniques of homogenization and filtering to build a novel scheme for solving multiscale inverse problems in an efficient and reliable manner. In the same spirit of \cite{AbD19, AbD20}, we prove that it is possible to eliminate the fast scales from the PDE appearing in the inverse problem relying on the theory of homogenization, thus obtaining a solution which is accurate in the vanishing limit for the multiscale parameter $\epl$. In our analysis, we both consider point-wise estimations as in \cite{ILS13} and Bayesian solutions as in \cite{ScS17}, thus showing convergence results which are endowed with decay rates under special assumptions on the problem. Inspired by \cite{CES14, CDS18, AbD20}, we then consider offline and online techniques for estimating the modelling error and prove a novel result indicating the computational cost which is required for such an estimation for any given multiscale problem.

In general, the EnKF has two main advantages with respect to other approaches. First, a Bayesian interpretation of the solution to the inverse problem is obtained from the algorithm without any additional cost. The Bayesian paradigm, frequently adopted in the context of inverse problems involving PDEs, provides a full uncertainty quantification on the solution and is therefore preferable to a point-wise estimation. Secondly, the EnKF is easily parallelizable, thus allowing in practice to solve complex inverse problems faster than employing, e.g., Markov chain Monte Carlo methods.

The main contributions of this paper are:
\begin{itemize}
\item to introduce a new method based on filtering techniques and numerical homogenization, which is computationally efficient and easy parallelizable to solve multiscale inverse problems;
\item to analyze theoretically the convergence properties of our method both from a point-wise and a Bayesian perspectives, proving the results of convergence of the EnKF scheme in the multiscale setting;
\item to estimate the modelling error caused by homogenization and discretization, and prove a novel theoretical results which strengthens its value in practice.
\end{itemize}

The outline of the work is the following. In Section \ref{Kalman} we briefly summarize the technique of ensemble Kalman inversion, show how it can be applied to multiscale inverse problems and state our main theoretical results. In Section \ref{Convergence} we present the analysis of our theoretical results, and Section \ref{Modelling} is dedicated to the estimation of the modelling error. Finally, in Section \ref{Experiments} we present a series of numerical experiments which corroborate our analysis.

\section{Ensemble Kalman inversion for multiscale problems}\label{Kalman}

In this section, we present the ensemble Kalman inversion technique for multiscale inverse problems. First, we introduce a generic framework and illustrate how the EnKF is employed to solve an inverse problem. Then we particularize to a inverse problems involving multiscale elliptic PDEs, and we conclude this section by announcing our main theoretical results. For a more exhaustive treatment of the EnKF in a generic PDE context, we refer the reader to \cite{ILS13, ScS17}.

\subsection{Ensemble Kalman inversion}
We first give a brief summary of the ensemble Kalman inversion for problems of the form
	\begin{equation}\label{inverse_problem}
\text{find } u \in X \text{ given observations } y = \mathcal G(u) + \eta \in Y,
\end{equation}
where $X$ and $Y$ are Hilbert spaces, the operator $\mathcal G \colon X \to Y$ is a generic forward map and the noise $\eta$ follows the Gaussian distribution $\eta \sim \mathcal N(0, \Gamma)$ with a symmetric positive definite covariance $\Gamma$. Kalman filters are traditionally employed to estimate the state of a dynamical system given partial and noisy observations of its state. In order to approximate the solution of the otherwise static problem \eqref{inverse_problem}, it is therefore natural to introduce some artificial dynamics. Let us consider the space $Z = X \times Y$ and the map $\Xi \colon Z \to Z$ given by
\begin{equation}
	\Xi(z) = \begin{bmatrix} u \\ \mathcal G(u) \end{bmatrix}, \quad \text{ for } \quad z = \begin{bmatrix} u \\ v \end{bmatrix} \in Z,
\end{equation}
Given an initial value $z_0 \in Z$, we define artificial discrete dynamics on $Z$ through the recursion
\begin{equation}\label{artificial_dynamics}
z_{n+1} = \Xi(z_n), \quad n = 0, 1, \ldots
\end{equation}
The dynamics on $Z$ are completed consistently with the problem \eqref{inverse_problem} by the observation equation
\begin{equation}\label{eq:obs_model}
y_{n+1} = H z_{n+1} + \eta_{n+1},
\end{equation} 
where $H \colon Z \to Y$ is the projection operator defined by $H = \begin{bmatrix} 0 & I \end{bmatrix}$ and $\{\eta_n \}_{n \in \mathbb{N}}$ is an i.i.d. sequence of random variables distributed identically to the noise of the inverse problem \eqref{inverse_problem}, i.e., $\eta_n \sim \mathcal N(0, \Gamma)$. In fact, let us remark that combining \eqref{artificial_dynamics} and \eqref{eq:obs_model} one gets $y_{n+1} = \mathcal G(u_n) + \eta_{n+1}$, which is in law equivalent to the equality appearing in \eqref{inverse_problem}.

Kalman filters proceed recursively to estimate the state of dynamics of the form \eqref{artificial_dynamics} when observations are provided by the model \eqref{eq:obs_model}. At each time $n$, the estimation is performed in two steps. First, equation \eqref{artificial_dynamics} is employed in the so-called \textit{prediction} step, and then \eqref{eq:obs_model} is employed to correct the prediction in the \textit{update} or \textit{analysis} step. In case $\Xi$ is a linear map, both prediction and update steps admit a closed-form expression, often referred to in literature as the Kalman formulae. Conversely, in case $\Xi$ is nonlinear, there exist no explicit solution to the estimation problem and one has to recur to an approximation such as the EnKF method, which we briefly describe here.

Given a positive integer $J$, the EnKF method proceeds by propagating and updating an ensemble $\{ z_n^{(j)} \}_{j = 1}^J \subset Z$ of particles with discrete approximations of the Kalman formulae. Let $\mathcal A \subset X$ be such that $\dim(\mathcal A) \leq J$, and let the initial ensemble $\{ z_0^{(j)} \}_{j = 1}^J$ to be given by
\begin{equation*}
z_0^{(j)} = \begin{bmatrix} \psi^{(j)} \\ \mathcal G(\psi^{(j)}) \end{bmatrix},
\end{equation*}
where $\{ \psi^{(j)} \}_{j = 1}^J \subset \mathcal A$. At each time $n = 0, 1, \ldots, N-1$, and for each $j = 1, \ldots, J$, the prediction step is simply given by
\begin{equation}\label{eq:EnKF_predict}
	\hat{z}_{n+1}^{(j)} = \Xi(z_n^{(j)}).
\end{equation}
In the analysis step, this partially-updated ensemble is updated given knowledge of the data $y$. For better exploring the space $Y$, the data is randomized and each particle $z_{n+1}^{(j)}$ is compared to i.i.d. versions of the data given by $y_{n+1}^{(j)} = y + \eta_{n+1}^{(j)}$, where $\eta_{n+1}^{(j)} \sim \mathcal N(0, \Gamma)$. The analysis step is then given by
\begin{equation}\label{KalmanUpdate}
	z_{n+1}^{(j)} = \hat{z}_{n+1}^{(j)} + K_{n+1} (y_{n+1}^{(j)} - H \hat{z}_{n+1}^{(j)}).
\end{equation}
The operator $K_{n+1} \colon Y \to Z$, the Kalman gain, weighs the effects of dynamics and observations in this two-step procedure, and is defined as
\begin{equation}\label{KalmanGain}
	K_{n+1} = C_{n+1} H^* R_{n+1}, \qquad R_{n+1} = (H C_{n+1} H^* + \Gamma)^{-1},
\end{equation}
where $C_{n+1}\colon Z \to Z$ is the empirical covariance of the partially-updated ensemble $\{\hat{z}_{n+1}^{(j)}\}_{j=1}^J$, the operator $H^*\colon Y \to Z$ is the adjoint of $H$, which is given in \eqref{eq:obs_model}, and we recall $\Gamma$ to be the covariance of the noise $y$, so that $R_{n+1} \colon Y \to Y$. Intuitively, one can notice that when the ensemble's covariance $C_{n+1}$ is large with respect to the noise covariance $\Gamma$, i.e., the observation model is more precise than the dynamics, we will have $z_{n+1}^{(j)} \approx y_{n+1}^{(j)}$, while in the opposite case we will have $z_{n+1}^{(j)} \approx \hat z_{n+1}^{(j)}$. A more precise definition of the operators appearing above will be given in Section \ref{Convergence}. At the final step $N$, we project the particles on the space $X$ and average the result to obtain the estimate
\begin{equation*}
	u_{\mathrm{EnKF}} = \frac{1}{J} \sum_{j = 1}^J H^{\perp} z_{N}^{(j)} = \frac{1}{J} \sum_{j = 1}^J u_{N}^{(j)},
\end{equation*}
where $H^{\perp} \colon Z \to X$ is defined by $H^{\perp} = \begin{bmatrix} I & 0 \end{bmatrix}$. The last detail missing to fully define the EnKF is its initialization, i.e., the choice of the space $\mathcal A$ defining the initial ensemble. We assume prior knowledge is available on the parameter $u \in X$ and that it is summarized by a probability measure $\mu_0$ on $X$. In this case, one can draw $J$ i.i.d. samples $\psi^{(j)}$ from $\mu_0$ and fix $\mathcal A = \mathrm{span}\{\psi^{(j)} \}_{j=1}^J$.

\begin{remark}\label{rem:comp_cost} The computational cost of the EnKF method is approximately equal to the number of evaluations of the forward operator, which  in a PDE framework dominates with respect to the algebraic operations needed in the analysis step. Therefore, the complexity of the algorithm is $\mathcal{O}(JN)$. Nonetheless, let us remark that the prediction step \eqref{eq:EnKF_predict} can be easily parallelized, since the forward operator is applied independently to each particle. Hence, for a reasonable number of particles (or a high number of computing units), we have that the overall cost is of order $\mathcal O(N)$.
\end{remark}

As shown in \cite{ScS17}, a slight modification of the EnKF algorithm allows to obtain with no additional cost a Bayesian solution to \eqref{inverse_problem} from the evolving ensemble. Let $\mu_0$ be, as above, a prior probability measure on $X$ and let the initial ensemble $\{\psi^{(j)}\}_{j=1}^J$ consist of i.i.d. samples from $\mu_0$. Given a number of steps $N$, let $\Delta = 1/N$ be a ``stepsize''. Let us modify the algorithm above by taking instead of the covariance $\Gamma$ of the noise its scaled version $\Delta^{-1}\Gamma$ in formula \eqref{KalmanGain}. Moreover, let us define the empirical measure $\hat \mu_n$ on $X$ induced by the ensemble at the $n$-th step, i.e.
\begin{equation}
	\hat \mu_n(du) = \frac1J \sum_{j=1}^J \delta_{u_n^{(j)}}(du),
\end{equation}
where $\delta_x$ is the Dirac mass concentrated in $x \in U$. Then, it has been shown in \cite{ScS17} that $\hat \mu_n$ is a good approximation of the measure $\mu_n$ defined by
\begin{equation}
	\mu_n(du) = \frac{1}{Z_n} e^{- n \Delta \Phi(u;y)} \mu_0(du),
\end{equation}
where $Z_n$ is the normalization constant and $\Phi(u;y)$ is the least squares functional
\[ \Phi(u;y) = \frac{1}{2} \norm{ \Gamma^{-1/2} (y - \mathcal{G}(u)) }_2^2. \]
For $n = N$, we have by definition $N\Delta = 1$ and the measure $\mu \defeq \mu_N$ given by
\begin{equation}\label{eq:posterior}
	\mu(du) = \frac{1}{Z} e^{-\Phi(u;y)} \mu_0(du),
\end{equation}
where $Z$ is the normalization constant, is exactly the posterior measure of the parameter $u$ given the prior $\mu_0$ in the Bayesian sense (see, e.g., \cite{Stu10}). Summarizing, if one carefully modifies formula \eqref{KalmanGain} for the Kalman gain, it is sufficient to run the EnKF method for $N$ steps and the empirical measure given by the particles is an approximation to the Bayesian posterior.

\subsection{Multiscale ensemble Kalman inversion}\label{sec:MultiscaleKalman}

In this work, we consider the application of ensemble Kalman inversion to a multiscale inverse problem of the form
\begin{equation} \label{multiscale_inverse_problem}
\text{find } u \in X \text{ given observations } y = \mathcal{G}^{\varepsilon}(u) + \eta \in Y,
\end{equation}
where $\epl > 0$ is the multiscale parameter, which often is $\epl \ll 1$, the operator $\mathcal G^\epl \colon X \to Y$ is the multiscale forward map and where, as above, $\eta \sim \mathcal N(0, \Gamma)$ for some symmetric positive definite covariance $\Gamma$ on $Y$. Let $\Omega \subset \R^d$ be an open bounded domain and let $H_0^1(\Omega)$ denote the space of functions $v \colon \Omega \to \R$ in $L^2(\Omega)$ with first order weak derivatives in $L^2(\Omega)$ and whose trace on $\partial \Omega$ vanishes. We consider the forward map $\mathcal G^{\epl}$ to be the composition $\mathcal G^\epl = \mathcal O \circ \mathcal S^\epl$ of an observation operator $\mathcal O \colon H_0^1(\Omega) \to Y$ and a multiscale solution operator $\mathcal S^\epl\colon X \to H_0^1(\Omega)$. In particular, for $u \in X$, the operator $\mathcal S^\epl \colon u \mapsto p^\epl \in H_0^1(\Omega)$ where $p^\epl$ is the weak solution of the elliptic PDE
\begin{equation}
\label{intro_problem_multiscale}
\left\{
\begin{alignedat}{2}
- \nabla \cdot ( A^{\varepsilon}_u \nabla p^{\varepsilon} ) &= f, \quad && \text{ in } \Omega, \\
p^{\varepsilon} &= 0, \quad && \text{ on } \partial \Omega,
\end{alignedat}
\right.
\end{equation}
for a right-hand side $f\in L^2(\Omega)$. We assume that the tensor $A^{\varepsilon}_u \colon \Omega \to \R^{d\times d}$ is a parametrized multiscale tensor admitting explicit scale separation between slow and fast spatial variables, i.e.,
\begin{equation*}
A^{\varepsilon}_u(x) = A \left ( u(x), \frac{x}{\varepsilon} \right ),
\end{equation*}
where the map $(t,x) \mapsto A(t, x/\epl)$ is assumed to be known and where $A$ is periodic in its second argument. In other words, the unknown $u$ of the inverse problem \eqref{multiscale_inverse_problem} governs the slow-scale variations of the rapidly-oscillating tensor $A_u^\epl$. 

Let us consider now the application of ensemble Kalman inversion to the inverse problem \eqref{multiscale_inverse_problem}. Since the PDE \eqref{intro_problem_multiscale} does not in general admit a closed-form solution, one has to employ a numerical approximation to evaluate the forward map $\mathcal G^\epl$. If $\varepsilon$ is small and we employ the finite element method (FEM), a fine discretization is needed to resolve the smallest scale and thus evaluate the forward operator $\mathcal{G}^{\varepsilon}$, which clearly leads to a high computational cost. Indeed, as for Remark \ref{rem:comp_cost}, a run of the EnKF algorithm would lead to $\mathcal O(N)$ solutions of \eqref{intro_problem_multiscale}, which is indeed unfeasible.
 
In order to approach the multiscale problem more efficiently we recur to the theory of homogenization (see e.g. \cite{CiD99}), which ensures the existence of a non-oscillating homogenized tensor $A^0_u$, such that for $\epl \to 0$ the solution $p^\epl$ of \eqref{intro_problem_multiscale} tends weakly in $H_0^1(\Omega)$ to the solution $p^0$ of the problem
\begin{equation}
\label{intro_problem_homogenized}
\left\{
\begin{alignedat}{2}
- \nabla \cdot ( A^0_u \nabla p^0 ) &= f, \quad && \text{ in } \Omega, \\
p^0 &= 0, \quad && \text{ on } \partial \Omega.
\end{alignedat}
\right.
\end{equation}
Hence, this homogenized problem is a good surrogate of \eqref{intro_problem_multiscale} when $\epl \ll 1$, and its non-oscillating nature allows us to discretize it with FEM on an arbitrarily coarse mesh, whose maximum diameter is denoted by $h$. Therefore, denoting by $\mathcal{G}^0_h \colon \mathcal{O} \circ \mathcal{S}^0_h$, where $\mathcal S^0_h \colon u \mapsto p_h^0$, the numerical solution of \eqref{intro_problem_homogenized}, we study in this paper the behavior of the EnKF when $\mathcal G^\epl$ is replaced by its cheap approximation $\mathcal G^0_h$. Let us denote by $\{u_{n,h}^{0,(j)}\}_{j=1}^J$ the ensemble obtained after $n$ iterations of the EnKF algorithm with the forward operators $\mathcal G_h^0$ in the prediction step \eqref{eq:EnKF_predict}. With this notation, given an initial ensemble $\{u_{0,h}^{0, (j)}\}_{j=1}^J$, at each step $n = 0, 1, \ldots, N-1$, our algorithm proceeds as
\begin{enumerate}
	\item for each $u_{n,h}^{0,(j)}$, compute the homogenized tensor $A_{u_n^{(j)}}^0$ and build the forward map $\mathcal G_h^0$,
	\item perform the prediction step \eqref{eq:EnKF_predict} with $\mathcal G_h^0$ and the analyis step \eqref{KalmanUpdate} to obtain the updated ensemble $\{u_{n+1,h}^{0,(j)}\}_{j=1}^J$.
\end{enumerate} 
The computation of the homogenized tensor relies as well on numerical procedures, here we use the finite element heterogeneous multiscale method (FE-HMM) \cite{Abd11, AEE12}. Let us finally remark that similar analyses have been carried on in \cite{NPS12, AbD19, AbD20} for different methodologies in the solution of \eqref{multiscale_inverse_problem}.

\subsection{Statement of main results}

Let us first introduce some assumptions and notation which will be employed in the analysis. First, we introduce a regularity assumption on tensors which will be fulfilled by $A^{\varepsilon}_u$ and $A^0_u$.
\begin{assumption} \label{ass_A} The tensor $A_u \colon \Omega \to \R^{d \times d}$ satisfies for all $u, u_1, u_2 \in X$ and $\xi \in \R^d$
\begin{equation}
\norm{A_{u_1} - A_{u_2}}_{L^{\infty}(\Omega;\R^{d \times d})} \le M \norm{u_1 - u_2}_X, \qquad A_u \xi \cdot \xi \ge \alpha_0 \norm{\xi}_2^2,
\end{equation}
where $M$ and $\alpha_0$ are positive constants. 
\end{assumption}
We now introduce a regularity assumption on the observation operator.
\begin{assumption}\label{ass_O}
	The observation operator $\mathcal{O} \colon H^1_0(\Omega) \to Y$ satisfies for all $p_1, p_2 \in H^1_0(\Omega)$
	\[ \norm{\mathcal{O}(p_1) - \mathcal{O}(p_2)}_Y \le C_{\mathcal O} \norm{p_1 - p_2}_{L^2(\Omega)}, \]
	where $C_{\mathcal O}$ is a positive constant.
\end{assumption}
Note that since $\mathcal O$ is defined on $H^1_0(\Omega) \subset L^2(\Omega)$, Assumption \ref{ass_O} is stronger than Lipschitz continuity. Finally, we introduce an assumption on the algorithm which will be employed in the analysis.
\begin{assumption} \label{ass_algo} All the particles in the ensemble lie at each iteration in a ball $B_R(u^*)$ for some $R > 0$ sufficiently big, where $u^*$ is the true value of the unknown.
\end{assumption}

For clarity, we present the analysis the finite-dimensional setting $X = \R^M$ and $Y = \R^L$ but claim that it can be readily generalized to the infinite-dimensional case. For an ensemble $u = \{u^{(j)}\}_{j=1}^J$ of particles in $\R^M$, we introduce the ensemble norm
\begin{equation} \label{ensemble_norm}
	\norm{u} \coloneqq \frac{1}{J} \sum_{j=1}^J \norm{u^{(j)}}_2,
\end{equation}
which is indeed a norm and where $\norm{\cdot}_2$ is the Euclidean norm in $\R^M$. Moreover, given a scalar $\alpha$, we define the linear combination $w = u + \alpha v$ between two ensembles $u$ and $v$ with the same number of particles $J$ as $\{w^{(j)} = u^{(j)} + \alpha v^{(j)}\}_{j=1}^J$.

We can now present the first main result of this work, in which we show the convergence of the ensemble obtained by the EnKF employing $\mathcal G^0_h$ to the one obtained employing the exact operator $\mathcal G^\epl$ linked to the PDE \eqref{intro_problem_multiscale}.

\begin{theorem} \label{convergence_result_full}
	Let $u_{N,h}^0 = \{ u_{N,h}^{0,(j)} \}_{j=1}^J$, $u_N^{\varepsilon} = \{ u_N^{\varepsilon,(j)} \}_{j=1}^J$ be the ensembles after $N$ iterations of the EnKF method with forward operators $\mathcal{G}^0_h$ and $\mathcal{G}^{\varepsilon}$ respectively. Then, if $A^\epl_u$ and $A^0_u$ satisfy Assumption \ref{ass_A} and if Assumption \ref{ass_O} and Assumption \ref{ass_algo} hold, we have
	\begin{equation*}
	\mathbb{E} \left [ \norm{u_N^{\varepsilon} - u_{N,h}^0} \right ] \to 0 \qquad \text{ as } \varepsilon, h \to 0.
	\end{equation*}
	In particular, if the exact solution $p^0$ of the homogenized problem \eqref{intro_problem_homogenized} is in $H^{q+1}(\Omega)$ with $q \ge 1$ and we employ polynomials of degree $r$ for the finite element basis, then
	\begin{equation*}
	\mathbb{E} \left [ \norm{u_N^{\varepsilon} - u_{N,h}^0} \right ] \le C ( \varepsilon + h^{s+1} ),
	\end{equation*}
	where $s = \min \{ r, q \}$ and $C > 0$ is a constant independent of $h$ and $\epl$.
\end{theorem}

The proof of this result is the main focus of Section \ref{sec:Conv_Point}. The second main theoretical result concerns the Bayesian interpretation of the EnKF methodology for inverse problems in the multiscale setting. Let $\mu_0$ be a prior measure on $X$ and the ensembles $u_{N,h}^0 = \{ u_{N,h}^{0,(j)} \}_{j=1}^J$, $u_N^{\varepsilon} = \{ u_N^{\varepsilon,(j)} \}_{j=1}^J$ resulting from the EnKF algorithms as in Theorem \ref{convergence_result_full} both initialized with an i.i.d. sample from $\mu_0$. We consider the discrete probability measures
\begin{equation}
\label{eq:ms_hom_posteriors}
\mu^{\varepsilon} = \frac{1}{J} \sum_{j=1}^J \delta_{u_N^{\varepsilon,(j)}} \qquad \text{ and } \qquad \mu_h^{0} = \frac{1}{J} \sum_{j=1}^J \delta_{u_{N,h}^{0,(j)}},
\end{equation}
i.e., the  EnKF approximations of the posterior $\mu$ on $u$ defined in \eqref{eq:posterior}. Our goal is providing a measure on how far the two measures are from each other with respect to $\epl$ and $h$. Let us remark that due to the randomization of the data at each step of the EnKF algorithm, both $\mu^\epl$ and $\mu_h^{0}$ are random probability measures. We now introduce the metric we consider for comparing the two measures.
\begin{definition}
	\label{weak_convergence_L1_distribution}
	Let $(\Omega, \mathcal{A}, P)$ be a probability space. A sequence of random measures $\{ \mu_n \}_{n \in \mathbb{N}}$ on a metric space $(E, \mathcal B(E))$ dependent on a random variable $\xi$ on $(\Omega, \mathcal{A}, P)$ is said to weakly converge in $L^1(\Omega)$ to a random measure $\mu$ on the same metric space if for all bounded continuous functions $f \in C^0_B(E)$ we have
	\[ \mathbb{E}_{\xi} \left [ \left |\int_E f \, d \mu_n - \int_E f \, d \mu  \right | \right ] \to 0. \]
	In this case we write $\mu_n \xrightharpoonup{L^1} \mu$.
\end{definition}

We can now state our second main result, whose proof is the main focus of Section \ref{sec:conv_Bayes}.
\begin{theorem}
	\label{convergence_posterior_distributions}
	Let the hypotheses of Theorem \ref{convergence_result_full} be satisfied. Then the sequence of random measures $\{\mu^\epl-\mu_h^0\}_{\epl, h}$, where $\mu^\epl$ and $\mu_h^0$ are defined in \eqref{eq:ms_hom_posteriors}, satisfies
	\[ \{\mu^{\varepsilon} - \mu_h^0\}_{\epl,h} \xrightharpoonup{L^1} 0 \qquad \text{ as } \varepsilon, h \to 0. \]
\end{theorem}
\begin{remark} It is possible to verify that in both Theorem \ref{convergence_result_full} and Theorem \ref{convergence_posterior_distributions} the limits with respect to $\epl$ and $h$ can be interchanged.
\end{remark}

\section{Convergence analysis}\label{Convergence}

In this section we prove Theorem \ref{convergence_result_full} and Theorem \ref{convergence_posterior_distributions}, the main results if this work. As announced above, the analysis is carried out in the finite dimensional case $X = \R^M$ and $Y = \R^L$, but it can be generalized to the infinite dimensional setting. For the purpose of the analysis, we introduce on top of the forward maps $\mathcal G^\epl$ and $\mathcal G^0_h$, which have been introduced in Section \ref{sec:MultiscaleKalman}, the operator $\mathcal G^0 = \mathcal O \circ \mathcal S^0$, where $\mathcal S^0\colon X \to H^1_0(\Omega)$ is the exact solution operator associated with the homogenized PDE \eqref{intro_problem_homogenized}.

\subsection{Convergence of the point estimate}\label{sec:Conv_Point}

We now focus on Theorem \ref{convergence_result_full}. It is clear from the desired bound that the effects of homogenization and discretization can be analysed separately. In particular, we first show the convergence of the ensemble generated employing the forward operator $\mathcal G^\epl$ to the one generated employing the exact homogenized operator $\mathcal G^0$ for $\epl \to 0$. Then, in an analogous fashion, we prove the convergence of the ensemble generated with $\mathcal G^0_h$ to the ensemble generated employing $\mathcal G^0$. In order to introduce a compact notation, we denote by $\mathcal U_{J,M}$ the set of ensembles of dimension $J$ with elements in $\R^M$ and we consider the homogenization error function $e \colon \R \times \mathcal U_{J,M} \to \R$, which is defined for a generic ensemble $u$ as
\begin{equation} \label{e}
	e(\varepsilon, u) = \frac{1}{J} \sum_{j=1}^J \norm{\mathcal{G}^{\varepsilon}(u^{(j)}) - \mathcal{G}^0(u^{(j)})}_2,
\end{equation}
and a discretization error function $\tilde{e} \colon \R \times \mathcal U_{J,M} \to \R$ as
\begin{equation} \label{e_tilde}
	\tilde e(h, u) = \frac{1}{J} \sum_{j=1}^J \norm{\mathcal{G}^0_h(u^{(j)}) - \mathcal{G}^0(u^{(j)})}_2.
\end{equation}
Before proving the main theorem, we introduce some preliminary results.

Let us first consider a generic forward operator involving an elliptic PDE and show that the associated forward map is Lipschitz continuous.
\begin{lemma} \label{G_lipschitz} Let $\mathcal{G} \colon \mathbb{R}^M \to \mathbb{R}^L$, $\mathcal{G} = \mathcal{O} \circ \mathcal{S}$ be a forward operator such that $\mathcal{O} \colon H^1_0(\Omega) \to \mathbb{R}^L$ is Lipschitz and $\mathcal{S} \colon \mathbb{R}^M \to H^1_0(\Omega)$, $\mathcal S\colon u \mapsto p$ is defined by the solution of
	\begin{equation} \label{problem_lemma}
	\left\{
	\begin{alignedat}{2}
	- \nabla \cdot ( {A_u \nabla p} ) &= f,  &&\quad \text{ in } \Omega, \\
	p &= 0, &&\quad \text{ on } \partial \Omega,
	\end{alignedat}
	\right.
	\end{equation}
	where $\Omega \subset \R^d$ is an open bounded set, the right-hand side $f \in L^2(\Omega)$ and the tensor $A_u$ satisfies Assumption \ref{ass_A}. Then $\mathcal G$ is Lipschitz with a constant depending only on the Poincaré constant of $\Omega$, on the constants $M$ and $\alpha$ appearing in Assumption \ref{ass_A}, on the right-hand side $f$ and on the Lipschitz constant of the operator $\mathcal O$.
\end{lemma}
The proof of Lemma \ref{G_lipschitz} is given in the Appendix. In the following Lemma, whose proof is also given in the Appendix, we consider the homogenization error defined in \eqref{e} and shows that it vanishes in the limit $\epl \to 0$.

\begin{lemma} \label{f_goes_to_0} Let $e$ be defined as \eqref{e}. Under Assumption \ref{ass_O}, we have for all $u \in \mathcal U_{J,M}$
\[ e(\varepsilon, u) \to 0 \quad \text{ as } \varepsilon \to 0. \]
Moreover, if the solution of the homogenized problem \eqref{intro_problem_homogenized} is in $H^2(\Omega)$ independently of $u$, then there exists $K > 0$ independent of $\epl$ and $u$ such that
\[ e(\varepsilon, u) \le K \varepsilon. \]
\end{lemma}

Finally, we consider the particle empirical covariances of ensembles given by the EnKF algorithm, thus proving their boundedness and Lipschitz continuity. The proof of this Lemma can be found in the Appendix.

\begin{lemma} \label{covariance_bound} Let $C^{up}(u) \in \R^{M\times L}$ and $C^{pp}(u) \in \R^{L \times L}$ be defined as
\begin{equation*}
	C^{up}(u) = \; \frac{1}{J} \sum_{j=1}^J \big(u^{(j)} - \bar{u}\big) \big(\mathcal{G}(u^{(j)}) - \bar{\mathcal{G}}\big)^T,  \qquad C^{pp}(u) = \; \frac{1}{J} \sum_{j=1}^J \big(\mathcal{G}(u^{(j)}) - \bar{\mathcal{G}}\big) \big(\mathcal{G}(u^{(j)}) - \bar{\mathcal{G}}\big)^T,
\end{equation*}
where $\bar u \in \R^M $ and $\bar{\mathcal G} \in \R^L$ are the empirical averages
\begin{equation*}
	\bar{u} = \; \frac{1}{J} \sum_{j=1}^J u^{(j)}, \qquad \bar{\mathcal{G}} = \; \frac{1}{J} \sum_{j=1}^J \mathcal{G}(u^{(j)}),
\end{equation*}
and let $\mathcal G \colon \R^M \to \R^L$ be Lipschitz with constant $C_{\mathcal G}$. Then, there exist four constants $C_i > 0$, $i = 1, \ldots, 4$, such that
\begin{multicols}{2}
\begin{enumerate}
\item $\norm{C^{up}(u)}_2 \le C_1$,
\item $\norm{C^{pp}(u)}_2 \le C_2$,
\item $\norm{C^{up}(u_1) - C^{up}(u_2)}_2 \le C_3 \norm{u_1 - u_2}$,
\item $\norm{C^{pp}(u_1) - C^{pp}(u_2)}_2 \le C_4 \norm{u_1 - u_2}$,
\end{enumerate}
\end{multicols}
for all ensembles $u, u_1, u_2 \in \mathcal U_{J,M}$ which are stable in the sense of Assumption \ref{ass_algo}.
\end{lemma}

In order to clarify the exposition, we first consider the amplification the error over one step between the EnKF algorithms employing the multiscale and the homogenized forward operators respectively, which is summarized in the following lemma.
\begin{lemma} \label{convergence_local} For all $n=0,\dots,N-1$, let $u_{n}^0 = \{ u_{n}^{0,(j)} \}_{j=1}^J$,$u_n^{\varepsilon} = \{ u_n^{\varepsilon,(j)} \}_{j=1}^J$ be the ensembles of particles at the $n$-th iteration of the EnKF for the forward operators $\mathcal{G}^0$ and $\mathcal{G}^{\varepsilon}$ respectively. Then, under Assumption \ref{ass_A}, Assumption \ref{ass_O} and Assumption \ref{ass_algo}, there exist positive constants $\alpha$ and $\gamma$ such that 
\begin{equation}
\mathbb{E} \left [ \norm{u_{n+1}^{\varepsilon} - u_{n+1}^0} \right ] \le \alpha \mathbb{E} \left [ \norm{u_n^{\varepsilon} - u_n^0} \right ] + \gamma \mathbb{E} \left [ e(\varepsilon, u_n^0) \right ],
\end{equation}
where $e(\epl, u)$ is given in \eqref{e}.
\end{lemma}
\begin{proof}
First, due to Assumption \ref{ass_O} and the Poincaré inequality with constant $C_p$ we have
\begin{equation*}
\norm{\mathcal{O}(p_1) - \mathcal{O}(p_2)}_2 \le C_{\mathcal O} \norm{p_1 - p_2}_{L^2(\Omega)} \le C_{\mathcal O} C_p \norm{\nabla p_1 - \nabla p_2}_{L^2(\Omega; \R^d)},
\end{equation*}
which shows that $\mathcal{O}$ is Lipschitz with constant $C_{\mathcal O}C_p$. Therefore, applying Lemma \ref{G_lipschitz}, we deduce that both $\mathcal{G}^0$ and $\mathcal{G}^{\varepsilon}$ are Lipschitz with constant $C_{\mathcal{G}}$ independent of $\varepsilon$. The Kalman update formulae \eqref{KalmanUpdate} restricted to the $u$ variable read (see \cite{ILS13})
\begin{align}
\label{stepn_e}
u_{n+1}^{\varepsilon,(j)} = & \; u_{n}^{\varepsilon,(j)} + C^{up}(u_n^{\varepsilon}) (C^{pp}(u_n^{\varepsilon}) + \Gamma)^{-1} (y_{n+1} - \mathcal{G}^{\varepsilon}(u_n^{\varepsilon,(j)})), \\
\label{stepn_0}
u_{n+1}^{0,(j)} = & \; u_{n}^{0,(j)} + C^{up}(u_n^{0}) (C^{pp}(u_n^{0}) + \Gamma)^{-1} (y_{n+1} - \mathcal{G}^{0}(u_n^{0,(j)})).
\end{align}
Combining \eqref{stepn_e} and \eqref{stepn_0}, we have
\begin{align*}
\mathbb{E} \left [ \norm{u_{n+1}^{\varepsilon} - u_{n+1}^0} \right ]  = \; \frac{1}{J} \sum_{j=1}^J \mathbb{E} &\left [ \left \lVert u_{n}^{\varepsilon,(j)} + C^{up}(u_n^{\varepsilon}) (C^{pp}(u_n^{\varepsilon}) + \Gamma)^{-1} (y_{n+1}^{(j)} - \mathcal{G}^{\varepsilon}(u_n^{\varepsilon,(j)})) \right. \right. \\
& \left. \left. -u_{n}^{0,(j)} - C^{up}(u_n^{0}) (C^{pp}(u_n^{0}) + \Gamma)^{-1} (y_{n+1}^{(j)} - \mathcal{G}^{0}(u_n^{0,(j)})) \right \rVert_2 \right ],
\end{align*}
and using the triangle inequality we obtain
\begin{equation}\label{decomp}
\mathbb{E} \left [ \norm{u_{n+1}^{\varepsilon} - u_{n+1}^0} \right ] \leq \E \left [ \norm{u_n^{\varepsilon} - u_n^0} \right ]  + S_1 + S_2 + S_3,
\end{equation}
where
\begin{align}
\label{2}
S_1 &= \; \frac{1}{J} \sum_{j=1}^J \mathbb{E} \left [ \norm{C^{up}(u_n^{\varepsilon}) - C^{up}(u_n^0)}_2 \norm{(C^{pp}(u_n^{\varepsilon}) + \Gamma)^{-1}}_2 \norm{y_{n+1}^{(j)} - \mathcal{G}^{\varepsilon}(u_n^{\varepsilon,(j)})}_2 \right ], \\ \label{3}
S_2 &= \; \frac{1}{J} \sum_{j=1}^J \mathbb{E} \left [ \norm{C^{up}(u_n^0)}_2 \norm{(C^{pp}(u_n^{\varepsilon}) + \Gamma)^{-1} - (C^{pp}(u_n^0) + \Gamma)^{-1}}_2 \norm{y_{n+1}^{(j)} - \mathcal{G}^{\varepsilon}(u_n^{\varepsilon,(j)})}_2 \right ], \\ \label{4}
S_3 &= \; \frac{1}{J} \sum_{j=1}^J \mathbb{E} \left [ \norm{C^{up}(u_n^0)}_2 \norm{(C^{pp}(u_n^0) + \Gamma)^{-1}}_2 \norm{\mathcal{G}^0(u_n^{0,(j)}) - \mathcal{G}^{\varepsilon}(u_n^{\varepsilon,(j)})}_2 \right ].
\end{align}
Let us introduce two useful inequalities which will be employed in the following. Given $A$ and $B$ square invertible matrices of the same size, it holds
\begin{equation}\label{difference_inverse}
\norm{A^{-1} - B^{-1}}_2 \le \norm{A^{-1}}_2 \norm{B^{-1}}_2 \norm{A - B}_2.
\end{equation} 
Moreover, if $A$ is positive semidefinite and $B$ is positive definite, it holds
\begin{equation}\label{inverse_sum}
\norm{(A + B)^{-1}}_2 \le \norm{B^{-1}}_2.
\end{equation}
Let us first consider $S_1$. Applying Lemma \ref{covariance_bound} and \eqref{inverse_sum} to the first two factors gives
\begin{equation}
	S_1 \leq \frac{C_3}{J} \sum_{j=1}^J \E\left[\norm{u_n^{\varepsilon} - u_n^0}\norm{\Gamma^{-1}}_2\norm{y_{n+1}^{(j)} - \mathcal{G}^{\varepsilon}(u_n^{\varepsilon,(j)})}_2\right].
\end{equation}
Moreover, since $y_{n+1}^{(j)} = y + \eta_{n+1}^{(j)}$ and since $y = \mathcal G^\epl(u^*) + \eta$, where $u^*$ is the true value of the unknown and $\eta$ is the true realization of the noise, the triangle inequality yields
\begin{equation}
\norm{y_{n+1}^{(j)} - \mathcal{G}^{\varepsilon}(u_n^{\varepsilon,(j)})}_2 \le \norm{\mathcal{G}^{\varepsilon}(u^*) - \mathcal{G}^{\varepsilon}(u_n^{\varepsilon,(j)})}_2 + \norm{\eta_{n+1}^{(j)} + \eta}_2,
\end{equation}
which, since $\mathcal{G}^{\varepsilon}$ is Lipschitz and due to Assumption \ref{ass_algo}, implies
\begin{equation}
\norm{y_{n+1}^{(j)} - \mathcal{G}^{\varepsilon}(u_n^{\varepsilon,(j)})}_2 \le C_{\mathcal{G}} R + \norm{\eta_{n+1}^{(j)} + \eta}_2.
\end{equation}
Hence, we get
\begin{equation*}
S_1 \leq \frac{1}{J} C_3 \norm{\Gamma^{-1}}_2 \sum_{j=1}^J \mathbb{E} \left [ \norm{u_n^{\varepsilon} - u_n^0} \left( C_{\mathcal{G}} R + \norm{\eta_{n+1}^{(j)} + \eta}_2 \right) \right ].
\end{equation*}
Finally, the random variables $\zeta_{n+1}^{(j)} \defeq \eta_{n+1}^{(j)} + \eta$ are i.i.d., distributed as $\zeta \sim \mathcal{N}(0,2\Gamma)$ and independent of $u_n^\epl$ and $u_n^0$, which implies first
\begin{equation*}
\mathbb{E}[\norm{\zeta}_2] \le \sqrt{\mathbb{E}[\norm{\zeta}_2^2]} = \sqrt{2\mathrm{tr}(\Gamma)},
\end{equation*}
and second, defining $\alpha_1 \defeq C_3 \norm{\Gamma^{-1}}_2 (C_{\mathcal{G}} R + \sqrt{2\mathrm{tr}(\Gamma)})$, yields the final bound
\begin{equation}
\label{bound2}
S_1 \leq \alpha_1 \mathbb{E} \left [ \norm{u_n^{\varepsilon} - u_n^0} \right ].
\end{equation}
Let us now consider the second term $S_2$. We apply Lemma \ref{covariance_bound} to the norm of $C^{up}(u_n^0)$. Moreover, applying the inequalities \eqref{difference_inverse}, \eqref{inverse_sum} and Lemma \ref{covariance_bound} gives
\begin{equation}
	\norm{(C^{pp}(u_n^{\varepsilon}) + \Gamma)^{-1} - (C^{pp}(u_n^0) + \Gamma)^{-1}}_2 \le C_4 \norm{\Gamma^{-1}}_2^2  \norm{u_n^{\varepsilon} - u_n^0}.
\end{equation}
Reasoning as for $S_1$ for the third factor appearing in \eqref{3} finally yields
\begin{equation}
\label{bound3}
S_2 \leq \alpha_2 \mathbb{E} \left [ \norm{u_n^{\varepsilon} - u_n^0} \right ],
\end{equation}
where $\alpha_2 \defeq C_1 C_4 \norm{\Gamma^{-1}}_2^2 (C_{\mathcal{G}} R + \sqrt{2\mathrm{tr}(\Gamma)})$. We now consider the last term $S_3$. The first factor appearing in \eqref{4} can be bounded by Lemma \ref{covariance_bound} and for the second factor we use \eqref{inverse_sum}, thus obtaining
\begin{equation*}
\norm{(C^{pp}(u_n^0) + \Gamma)^{-1}}_2 \le \norm{\Gamma^{-1}}_2.
\end{equation*}
Regarding the third factor of \eqref{4}, we apply the triangle inequality and the Lipschitz continuity of the forward operator $\mathcal{G}^{\varepsilon}$, which yield
\begin{equation}
\norm{\mathcal{G}^0(u_n^{0,(j)}) - \mathcal{G}^{\varepsilon}(u_n^{\varepsilon,(j)})}_2 \le \norm{\mathcal{G}^0(u_n^{0,(j)}) - \mathcal{G}^{\varepsilon}(u_n^{0,(j)})}_2 + C_{\mathcal{G}} \norm{u_n^{0,(j)} - u_n^{\varepsilon,(j)}}_2.
\end{equation}
Substituting back into $S_3$ and by definition of $e(\epl, u_n^0)$ and of the ensemble norm we obtain
\begin{equation*}
S_3 \leq C_1 \norm{\Gamma^{-1}}_2 \mathbb{E} \left [ e(\varepsilon, u_n^0) \right ] + C_1 \norm{\Gamma^{-1}}_2 C_{\mathcal{G}} \mathbb{E} \left [ \norm{u_n^0 - u_n^{\varepsilon}} \right ].
\end{equation*}
Therefore, defining $\alpha_3 = C_1 \norm{\Gamma^{-1}}_2 C_{\mathcal{G}}$ and $\gamma = C_1 \norm{\Gamma^{-1}}_2$ we have the bound
\begin{equation}
\label{bound4}
S_3 \leq \alpha_3 \mathbb{E} \left [ \norm{u_n^0 - u_n^{\varepsilon}} \right ] + \gamma \mathbb{E} \left [ e(\varepsilon, u_n^0) \right ].
\end{equation}
Finally, defining $\alpha \defeq 1 + \alpha_1 + \alpha_2 + \alpha_3$, and using the results \eqref{decomp}, \eqref{bound2}, \eqref{bound3} and \eqref{bound4}, we obtain the desired result.
\end{proof}

We now present the main result about global multiscale convergence of the EnKF algorithm.

\begin{proposition} \label{convergence_result}
Under the notation and assumptions of Lemma \ref{convergence_local}, letting $u_0^{\varepsilon} = u_0^0$ be the same initial ensemble, we have
\[ \mathbb{E} \left [ \norm{u_N^{\varepsilon} - u_N^0} \right ] \to 0 \qquad \text{ as } \varepsilon \to 0.  \]
Moreover, if the solution of the homogenized problem \eqref{intro_problem_homogenized} is sufficiently regular, namely $p^0 \in H^2(\Omega)$, then there exists $K_1 > 0$ independent of $\epl$ such that
\[ \mathbb{E} \left [ \norm{u_N^{\varepsilon} - u_N^0} \right ] \le K_1 \varepsilon. \]
\end{proposition}
\begin{proof}
	Since $u_0^\epl = u_0^0$, iterating the estimate of Lemma \ref{convergence_local} yields
	\begin{equation*}
		\mathbb{E} \left [ \norm{u_N^{\varepsilon} - u_N^0} \right ] \le \gamma \sum_{i=0}^{N-1} \alpha^{N-1-i} \mathbb{E} \left [ e(\varepsilon, u^0_i) \right ].
	\end{equation*}
	Applying Lemma \ref{f_goes_to_0}, we have $e(\varepsilon, u_i^0) \to 0$ for all $i = 0, \dots, N-1$, hence as $\varepsilon \to 0$
	\begin{equation*}
		\mathbb{E} \left [ \norm{u_N^{\varepsilon} - u_N^0} \right ] \to 0.
	\end{equation*}
	Moreover, if $p^0$ belongs to $H^2(\Omega)$, applying Lemma \ref{f_goes_to_0} gives
	\begin{equation}
		\mathbb{E} \left [ \norm{u_N^{\varepsilon} - u_N^0} \right ] \le K_1 \varepsilon,
	\end{equation}
	where $K_1 = \gamma (\alpha^N - 1) K /(\alpha - 1)$, which is the desired result.
\end{proof}

We now consider convergence with respect to the FEM discretization of the homogenized problem. First, we introduce a preliminary result, which plays the role of Lemma \ref{f_goes_to_0} in the context of numerical convergence and whose proof is given in the Appendix.
\begin{lemma}
	\label{fh_goes_to_0}
	Let $\tilde{e}$ be defined in \eqref{e_tilde} and let Assumption \ref{ass_O} hold. If the exact solution $p^0$ of the homogenized problem \eqref{problem_lemma} is in $H^{q+1}(\Omega)$, the right-hand side $f$ is in $H^{q-1}(\Omega)$ and we employ polynomials of degree $r$ for the finite element basis, then
	\[ \tilde{e}(h, u) \le \tilde{K} h^{s+1}, \]
	where $s = \min \{ r, q \}$.
\end{lemma}
We can now state the main result concerning convergence with respect to the numerical discretization of the homogenized problem.
\begin{proposition}
\label{convergence_result_h}
Let $u_{N}^0 = \{ u_{N}^{0,(j)} \}_{j=1}^J$, $u_{N,h}^0 = \{ u_{N,h}^{0,(j)} \}_{j=1}^J$ be the ensembles of particles at the last iteration of the iterative ensemble Kalman filter for the forward operators $\mathcal{G}^0$ and $\mathcal{G}_h^0$ respectively. Then, under  Assumption \ref{ass_A}, Assumption \ref{ass_O}, Assumption \ref{ass_algo} and if the exact solution $p^0$ of the homogenized problem \eqref{problem_lemma} is in $H^{q+1}(\Omega)$ and we use polynomials of degree $r$ for the finite element basis, we have
\[ \mathbb{E} \left [ \norm{u_{N,h}^0 - u_N^0} \right ] \le K_2 h^{s+1}, \]
where $s = \min \{ r, q \}$ and $K_2$ is a positive constant independent of $h$.
\end{proposition}
\begin{proof}
The proof of Proposition \ref{convergence_result_h} is identical to the proof of Proposition \ref{convergence_result}, except that all the ensembles $\{ u_n^{\varepsilon} \}_{n=1}^N$ obtained by the multiscale operator $\mathcal{G}^{\varepsilon}$ have to be replaced by the ensembles $\{ u_{n,h}^0 \}_{n=1}^N$ obtained by the finite element discretization of the homogenized operator $\mathcal{G}^0_h$. Moreover Lemma \ref{f_goes_to_0} for the error $e$ has to be replaced by Lemma \ref{fh_goes_to_0} for the error $\tilde{e}$.
\end{proof}

Applying Proposition \ref{convergence_result} and Proposition \ref{convergence_result_h}, we finally prove Theorem \ref{convergence_result_full}.
\begin{proof}[Proof of Theorem \ref{convergence_result_full}]
An application of the triangle inequality yields
\[ \mathbb{E}[ \norm{u_N^{\varepsilon} - u_{N,h}^0} ] \le \mathbb{E} [ \norm{u_N^{\varepsilon} - u_N^0} ] + \mathbb{E} [ \norm{u_N^0 - u_{N,h}^0} ]. \]
The two addends can be bounded applying Proposition \ref{convergence_result} and Proposition \ref{convergence_result_h}, thus obtaining the desired result for $C = \max \{ K_1, K_2 \}$.
\end{proof}

\subsection{Convergence of the posterior distributions}\label{sec:conv_Bayes}

In this section, we give the proof of Theorem \ref{convergence_posterior_distributions}, i.e., the convergence of the discrete posterior measures $\mu^\epl$ to $\mu^0_h$ introduced in \eqref{eq:ms_hom_posteriors} as $\epl, h \to 0$. Let $u^* \in \R^M$ and let $B_R(u^*)$ be the ball of radius $R$ centered in $u^*$ with respect to the norm $\norm{\cdot}_s$ with $s \in [1,\infty]$. Due to the discrete nature of these distributions, we study convergence with respect to the Wasserstein metrics, for which we report its standard definition in the metric spaces $(B_R(u^*), \norm{\cdot}_s)$, which can be found, e.g., in \cite{San15}.

\begin{definition}
\label{Wasserstein_definition}
Let $\mu$ and $\nu$ be two probability measures on the metric space $(B_R(u^*), \norm{\cdot}_s)$. The Wasserstein distance between $\mu$ and $\nu$ is defined for all $p \in [1, \infty)$ as
\begin{equation}
\label{Wasserstein_distance}
W_{p,s}(\mu, \nu) = \left ( \inf_{\gamma \in \Gamma(\mu, \nu)} \int_{B_R(u^*) \times B_R(u^*)} \norm{u - v}_s^p d \gamma(u,v) \right )^{1/p},
\end{equation}
where $\Gamma(\mu, \nu)$ denotes the collection of all joint distributions on $B_R(u^*) \times B_R(u^*)$ with marginals $\mu$ and $\nu$ on the first and second factors respectively.
\end{definition}

\begin{remark} If $\mu$ and $\nu$ are two discrete distributions on finite state spaces, respectively $\Omega_1 = \{ u_1, \dots, u_{K_1} \}$ and $\Omega_2 = \{ v_1, \dots, v_{K_2} \}$ included in $B_R(u^*)$, then \eqref{Wasserstein_distance} can be written as
\begin{equation}
\label{Wasserstein_distance_discrete}
W_{p,s}(\mu, \nu) = \left ( \inf_{\gamma \in \R^{K_1 \times K_2}} \sum_{i=1}^{K_1} \sum_{j=1}^{K_2} \norm{u_i - v_j}_s^p \gamma_{ij} \right )^{1/p},
\end{equation}
where the matrix $\gamma$ has to satisfy the following constraints
\begin{equation}\label{eq:Wass_constraints}
\sum_{j=1}^{K_2} \gamma_{ij} = \mu(u_i) \quad \text{for all } i = 1, \dots K_1, \qquad \sum_{i=1}^{K_1} \gamma_{ij} = \nu(v_j) \quad \text{for all } j = 1, \dots K_2.
\end{equation}
\end{remark}

We now show that the distance $W_{1,2}$ is bounded by the distance induced by the ensemble norm defined in \eqref{ensemble_norm}. This result will be crucial later to prove Theorem \ref{convergence_result_full}.

\begin{lemma}
\label{Wasserstein_ensemble_norm}
Let $u_1 = \{ u_1^{(j)} \}_{j=1}^J$, $u_2 = \{ u_2^{(j)} \}_{j=1}^J$ be two ensembles of particles and let $\mu_1, \mu_2$ be the corresponding distributions defined as sum of Dirac masses
\begin{equation*}
\mu_1 = \frac{1}{J} \sum_{j=1}^J \delta_{u_1^{(j)}}, \qquad \qquad \mu_2 = \frac{1}{J} \sum_{j=1}^J \delta_{u_2^{(j)}}.
\end{equation*}
Then for all $s\in[1,\infty]$ and $p\in[1,\infty)$ it holds
\[ W_{p,s}(\mu_1, \mu_2) \le \left ( \frac{1}{J} \sum_{j=1}^J \norm{u_1^{(j)} - u_2^{(j)}}_s^p \right )^{\frac{1}{p}} \]
and, in particular,
\[ W_{1,2}(\mu_1, \mu_2) \le \norm{u_1 - u_2}. \]
\end{lemma}
\begin{proof}
Take $\gamma^*$ defined as
\begin{equation*}
\gamma^*(u_1^{(j)}, u_2^{(i)}) =
\begin{cases}
\frac{1}{J} & \text{ if } i = j \\
0 & \text{ if } i \neq j,
\end{cases}
\end{equation*}
which satisfies the constraints \eqref{eq:Wass_constraints}, and note that
\begin{equation*}
\sum_{j=1}^{J} \sum_{i=1}^{J} \norm{u_1^{(j)} - u_2^{(i)}}_s^p \gamma^*(u_1^{(j)}, u_2^{(i)}) = \frac{1}{J} \sum_{j=1}^{J} \norm{u_1^{(j)} - u_2^{(j)}}_s^p.
\end{equation*}
Therefore, by definition of Wasserstein distance for discrete distributions on finite spaces \eqref{Wasserstein_distance_discrete}, we deduce that
\[ W_{p,s}(\mu_1, \mu_2) \le \left ( \frac{1}{J} \sum_{j=1}^J \norm{u_1^{(j)} - u_2^{(j)}}_s^p \right )^{\frac{1}{p}}, \]
which is the desired result. Finally, taking $p = 1$ and $s = 2$ and recalling the ensemble norm defined in \eqref{ensemble_norm}, we obtain the second inequality.
\end{proof}

We now analyze the relationship between the weak $L^1$ convergence introduce in Definition \ref{weak_convergence_L1_distribution} and the convergence with respect to the expectation of the Wasserstein distance for random probability measures. In particular, we prove that the latter implies the former, which was already proved in \cite{San15} for non-random measures. Here, we extend the result to random probability measures. The proof of the following Lemma is given in the Appendix.
\begin{lemma}
\label{equivalence_convergence_DW1}
Let $(\Omega, \mathcal{A}, P)$ be a probability space. Let the sequence $\{ \mu_n \}_{n \in \mathbb{N}}$ and $\mu$ be random probability measures on the metric space $(B_R(u^*), \norm{\cdot}_s)$ dependent on the random variable $\xi$ on $(\Omega, \mathcal{A}, P)$. If
\[ \mathbb{E}_{\xi} [W_{1,s}(\mu_n, \mu)] \to 0, \]
then $\mu_n \xrightharpoonup{L^1} \mu$.
\end{lemma}

We can now complete the proof of Theorem \ref{convergence_posterior_distributions}.

\begin{proof}[Proof of Theorem \ref{convergence_posterior_distributions}]
Applying Lemma \ref{Wasserstein_ensemble_norm} and due to Theorem \ref{convergence_result_full}, we deduce that for $\epl, h \to 0$ it holds
\[ \mathbb{E} [ W_{1,2}(\mu^{\varepsilon}, \mu_h^0) ] \to 0. \]
Note that the only difference in the update step of the EnKF when used for a point estimate and in the Bayesian framework is that $\Gamma$ is replaced by $\Delta^{-1} \Gamma$ where $\Delta = 1/N$. The constants of the proof of Theorem \ref{convergence_result_full} depend on $\norm{\Gamma^{-1}}_2$, which is now replaced by $\norm{(\Delta^{-1} \Gamma)^{-1}}_2$, which can be bounded by $\norm{\Gamma^{-1}}_2$ as
\[ \norm{(\Delta^{-1} \Gamma)^{-1}}_2 = \Delta \norm{\Gamma^{-1}}_2 \le \norm{\Gamma^{-1}}_2. \]
Finally, applying Lemma \ref{equivalence_convergence_DW1}, we obtain the desired result.
\end{proof}

\section{Modelling error}\label{Modelling}
In this section, we consider the effects of model misspecification due to the homogenization and discretization error. All the results presented in Section \ref{Convergence} deal with the asymptotic case $h, \epl \to 0$, which is unrealistic in applications. Let us recall that the original inverse problem involves predicting the exact unknown $u^*$ from observations originated by the model
\begin{equation}
\label{model_error_y}
y = \mathcal{G}^{\varepsilon}(u^*) + \eta,
\end{equation}
where $\eta \sim \mathcal{N}(0,\Gamma)$ is the noise. Since evaluating $\mathcal G^\epl$ is too expensive and in many applications unfeasible, we wish to employ the cheaper forward operator $\mathcal{G}^0_h$. Hence, we rewrite \eqref{model_error_y} as
\begin{equation}
\label{model_error_y2}
y = \mathcal{G}^0_h(u^*) + \mathcal{E}(u^*) + \eta,
\end{equation}
where
\[ \mathcal{E}(u^*) \defeq \mathcal{G}^{\varepsilon}(u^*) - \mathcal{G}^0_h(u^*). \]
The quantity $\mathcal{E}(u^*)$ represents the error introduced by misspecification of the forward model. Equation \eqref{model_error_y2} shows that the observed data $y$ can be seen as data originating by the discrete homogenized model which is affected by two sources of errors, the original noise and the modelling error. This formulation of modelling error was originally presented in \cite{CES14}, and then applied to multiscale inverse problems in \cite{AbD20}. Following \cite{CES14, AbD20}, we assume that the modelling error is a Gaussian random variable independent of the noise $\eta$, so that $\mathcal{E} \sim \mathcal{N}(m, \Sigma)$ for all $u$, and write
\begin{equation}
\label{model_error_y3}
y = \mathcal{G}^0_h(u^*) + m + \zeta + \eta,
\end{equation}
where $\zeta \sim \mathcal{N}(0, \Sigma)$. There is no theoretical guarantee for the modelling error to be distributed as a Gaussian in this framework. Nevertheless, it has been shown in \cite{NoP09} that in the one-dimensional case a Gaussian assumption can be employed effectively for the modelling error, thus partially justifying our choice. Then we define
\[ \tilde{y} = y - m \qquad \text{and} \qquad \tilde{\eta} = \eta + \zeta \sim \mathcal{N}(0, \Gamma + \Sigma) \]
and, from \eqref{model_error_y3}, we obtain
\begin{equation}
\label{model_error_y4}
\tilde{y} = \mathcal{G}^0_h(u^*) + \tilde{\eta}.
\end{equation}
Therefore, if the mean $m$ and covariance $\Sigma$ of the modelling error are known, a more reliable approximation of the unknown $u^*$ can be obtained applying the EnKF to \eqref{model_error_y4}. The modelling error distribution, by assumption fully determined by its mean and covariance, is approximated offline. We sample $N_{\mathcal{E}}$ unknowns $\{ u_i \}_{i=1}^{N_\mathcal{E}}$ from $\mu_0$ and, for all $i = 1, \dots, N_{\mathcal{E}}$, we apply both the forward operators $\mathcal{G}^{\varepsilon}(u_i)$ and $\mathcal{G}^0_h(u_i)$. Then we compute
\[ \mathcal{E}_i = \mathcal{G}^{\varepsilon}(u_i) - \mathcal{G}^0_h(u_i), \]
and the mean $m$ and the covariance $\Sigma$ are obtained as the empirical mean and covariance of the sample $\{\mathcal E_i\}_{i=1}^{N_{\mathcal E}}$. This procedure is computationally involved due to the multiple evaluations of $\mathcal G^\epl$, but it has to be performed only once and can then be applied to different sets of observations and true values $u^*$. Let us also remark that on the one hand, due to the theory of homogenization, the modelling error can be considered negligible when $\epl$ is very small, and the expensive estimation of $\mathcal E$ may not be necessary. On the other hand, when $\epl$ is larger, the homogenized equation does not provide with a good approximation of the multiscale problem, and an estimation of $\mathcal E$ is required. One may rightfully argue that in case $\epl = \mathcal O(1)$, it is possible to evaluate the forward operator $\mathcal G^\epl$ without a large computational effort. Hence, the techniques presented in this section are relevant for mid-range values of $\epl$, for which $\mathcal E$ is significant with respect to the noise $\eta$. Moreover, we remarked in practice via numerical experiments that a small number $N_{\mathcal E}$ can be employed to obtain a satisfactory approximation of the modelling error. A theoretical justification of this property is provided by Theorem \ref{inequality_modelling_error} and Theorem \ref{inequality2_modelling_error}.

In order to obtain a more reliable approximation of the distribution of the modelling error, we can follow a dynamic approach based on the estimation of the mean $m$ and the covariance $\Sigma$ online, i.e., during the run of the EnKF algorithm. This methodology has been developed in \cite{CDS18}. In particular, we sequentially apply the ensemble Kalman method for $\mathcal{L}$ levels and, at each level $\ell = 1, \dots, \mathcal{L}$, we update the distribution of the modelling error, which is denoted by $\nu^{\ell} = \mathcal{N}(m^{\ell}, \Sigma^{\ell})$. Letting
\[ \mu_n^{\ell} = \frac{1}{J} \sum_{j=1}^J \delta_{u^{\ell (j)}_n} \]
be the approximation of the distribution of the particles at iteration $n$ at level $\ell$, $\mu_0^{\ell + 1} = \mu_{N^{\ell}}^{\ell}$ and $\mu_0^1 = \mu_0$, where $N^{\ell}$ is the number of iterations at level $\ell$, then the mean $m^{\ell}$ and the covariance $\Sigma^{\ell}$ are approximated as in the offline approach with the only difference that $\mu_0$ is replaced by $\mu_0^{\ell}$. This approach provides indeed a better approximation of the modelling error as instead of taking the samples from the prior distribution, they are drawn from distributions which are progressively closer to the true posterior. On the other hand, this procedure has to be done online and it is computationally expensive because it requires the resolution of $N_{\mathcal{E}} = \sum_{\ell=1}^{\mathcal{L}} N_{\mathcal{E}}^{\ell}$ full multiscale problems.

Finally, we are interested in studying whether the simple offline method for estimating the modelling error provides indeed a good approximation. In this direction, we give in Theorem \ref{inequality_modelling_error} and Theorem \ref{inequality2_modelling_error} a criterion on how to choose the number $N_{\mathcal{E}}$ of full multiscale problems which has to be solved in order to have a reliable approximation of the true mean $m^*$ and covariance $\Sigma^*$ of the modelling error with respect to $\epl$ and $h$. Before stating Theorem \ref{inequality_modelling_error} and Theorem \ref{inequality2_modelling_error}, let us recall the Hoeffding's and McDiarmid's inequalites, which will be used in the proofs. Let $\{ Y_i \}_{i=1}^N$ be independent random variables with values in $[a,b]$, and let $\bar{Y}$ be the sample average of $\{Y_i\}_{i=1}^N$. Then, the Hoeffding's inequality states that for all $\eta \in \R$ it holds
\[ \mathbb{P} (\abs{\bar{Y} - \mathbb{E}[Y]} \ge \eta) \le 2 \exp \left \{ - \frac{2 \eta^2 N}{(b-a)^2} \right \}. \]
Moreover, let $\{ X_i \}_{i=1}^N$ be independent random variables with values in the space $\mathcal{X}$, and let $\phi \colon \mathcal{X}^N \to \R$ satisfy for all $i = 1, \dots, N$
\begin{equation}
\sup_{x_1, \dots, x_N, \hat{x}_i} \abs{\phi(x_1, \dots, x_{i-1}, x_i, x_{i+1}, \dots, x_N) - \phi(x_1, \dots, x_{i-1}, \hat{x}_i, x_{i+1}, \dots, x_N)} \le c,
\end{equation}
then the McDiarmid's inequality states that for all $\eta \in \R$ it holds
\begin{equation}
\mathbb{P} ( \abs{\phi(X_1, \dots, X_N) - \mathbb{E}[\phi(X_1, \dots, X_N)]} \ge \eta ) \le 2 \exp \left \{ - \frac{2 \eta^2}{N c^2} \right \}.
\end{equation}

\begin{theorem}
	\label{inequality_modelling_error}
	Let $\alpha \in (0,1)$, $\eta > 0$ and $C_{\mathcal{E}} = \max \{ K, \tilde{K} \}$, where $K$ and $\tilde{K}$ are the constants of Lemma \ref{f_goes_to_0} and Lemma \ref{fh_goes_to_0}. Let $\{ \mathcal{E}_i \}_{i=1}^{N_{\mathcal{E}}} \subset \R^L$ be given by
	\[ \mathcal{E}_i = \mathcal{G}^{\varepsilon}(u_i) - \mathcal{G}^0_h(u_i) \qquad \text{for all } i = 1, \dots, N_{\mathcal{E}}, \]
	for a sample of realizations $\{ u_i \}_{i=1}^{N_{\mathcal{E}}}$ from the standard normal distribution $\mathcal{N}(0,I)$, let $m$ be the sample mean of $\{\mathcal E_i\}_{i=1}^{N_{\mathcal E}}$ and $m^* = \mathbb{E}[\mathcal{E}_i]$. If
	\[ N_{\mathcal{E}} \ge 4 C_{\mathcal{E}}^2 \frac{L}{\eta^2} \log \left ( \frac{2L}{\alpha} \right ) \left [ \varepsilon^2 + h^{2(s+1)} \right ], \]
	where $s$ is given by Lemma \ref{fh_goes_to_0}, then
	\[ \mathbb{P} \left ( \norm{m - m^*}_2 \le \eta \right ) \ge 1 - \alpha. \]
\end{theorem}
\begin{proof}
	First, note that the modelling error is bounded, indeed by Lemma \ref{f_goes_to_0} and Lemma \ref{fh_goes_to_0}, we have for each $i = 1, \dots, N_{\mathcal{E}}$
	\begin{equation*}
	\norm{\mathcal{E}_i}_2 = \norm{\mathcal{G}^{\varepsilon}(u_i) - \mathcal{G}^0_h(u_i)}_2 \le \norm{\mathcal{G}^{\varepsilon}(u_i) - \mathcal{G}^0(u_i)}_2 + \norm{\mathcal{G}^0(u_i) - \mathcal{G}^0_h(u_i)}_2 \le K \varepsilon + \tilde{K} h^{s+1},
	\end{equation*}
	so each component $(\mathcal{E}_i)_l$, for $l = 1, \dots, L$, is bounded by the same constant
	\begin{equation}
	\label{bound_component}
	\abs{(\mathcal{E}_i)_l} \le \norm{\mathcal{E}_i}_2 \le K \varepsilon + \tilde{K} h^{s+1} \le C_{\mathcal{E}} ( \varepsilon + h^{s+1} ).
	\end{equation}
	Observe that if
	\begin{equation}
	\abs{m_l - m^*_l} \le \frac{\eta}{\sqrt{L}} \qquad \text{for each } l = 1, \dots, L,
	\end{equation}
	then
	\begin{equation}
	\norm{m - m^*}_2 = \left ( \sum_{l=1}^L \abs{m_l - m^*_l}^2 \right )^{\frac{1}{2}} \le \eta,
	\end{equation}
	which implies that
	\begin{equation}
	\label{1toL}
	\mathbb{P}( \norm{m - m^*}_2 \le \eta ) \ge \mathbb{P} \left ( \abs{m_l - m^*_l} \le \frac{\eta}{\sqrt{L}} \quad \forall \; l = 1, \dots, L \right ).
	\end{equation}
	Using \eqref{bound_component} and applying Hoeffding's inequality we have
	\begin{equation}
	\label{hoeffding}
	\mathbb{P} \left ( \abs{m_l - m^*_l} \ge \frac{\eta}{\sqrt{L}} \right ) \le 2 \exp \left \{ - \frac{2 \eta^2 N_{\mathcal{E}}}{4 L C_{\mathcal{E}}^2 (\varepsilon + h^{s+1})^2 } \right \} \le 2 \exp \left \{ - \frac{\eta^2 N_{\mathcal{E}}}{4 L C_{\mathcal{E}}^2 (\varepsilon^2 + h^{2(s+1)})} \right \}.
	\end{equation}
	Define the events $A_l = \left \{ \abs{m_l - m^*_l} \le \frac{\eta}{\sqrt{L}} \right \}$ for each $l = 1, \dots, L$, then we have
	\begin{equation*}
	\mathbb{P} \left ( \abs{m_l - m^*_l} \le \frac{\eta}{\sqrt{L}} \quad \forall \; l = 1, \dots, L \right ) = \mathbb{P} \left ( \bigcap_{l=1}^L A_l \right ),
	\end{equation*}
	and, applying the De Morgan's laws and the union bound, we obtain
	\begin{equation}
	\label{de_morgan_union_bound}
	\mathbb{P}\left ( \bigcap_{l=1}^L A_l \right ) = 1 - \mathbb{P} \left ( \left ( \bigcap_{l=1}^L A_l \right )^C \right ) = 1 - \mathbb{P} \left ( \bigcup_{l=1}^L A_l^C \right ) \ge 1 - \sum_{l=1}^L \mathbb{P}(A_l^C).
	\end{equation}
	Therefore, thanks to \eqref{1toL}, \eqref{hoeffding} and \eqref{de_morgan_union_bound}, we have
	\begin{equation}
	\mathbb{P}( \norm{m - m^*}_2 \le \eta ) \ge 1 - L \max_{l=1,\dots,L} \mathbb{P} \left ( \abs{m_l - m^*_l} \ge \frac{\eta}{\sqrt{L}} \right ) \ge 1 - 2L \exp \left \{ - \frac{\eta^2 N_{\mathcal{E}}}{4 L C_{\mathcal{E}}^2 (\varepsilon^2 + h^{2(s+1)})} \right \},
	\end{equation}
	and if $N_{\mathcal{E}}$ satisfies the hypothesis we obtain the desired result.
\end{proof}

\begin{theorem}
	\label{inequality2_modelling_error}
	Let $\alpha \in (0,1)$, $\eta > 0$ and $C_{\mathcal{E}} = \max \{ K, \tilde{K} \}$, where $K$ and $\tilde{K}$ are the constants of Lemma \ref{f_goes_to_0} and Lemma \ref{fh_goes_to_0}. Let $\{ \mathcal{E}_i \}_{i=1}^{N_{\mathcal{E}}} \subset \R^L$ be given by
	\[ \mathcal{E}_i = \mathcal{G}^{\varepsilon}(u_i) - \mathcal{G}^0_h(u_i) \qquad \text{for all } i = 1, \dots, N_{\mathcal{E}}, \]
	for a sample of realizations $\{ u_i \}_{i=1}^{N_{\mathcal{E}}}$ from the standard normal distribution $\mathcal{N}(0,I)$, let $m$ and $\Sigma$ be the sample mean and covariance of $\{\mathcal E_i\}_{i=1}^{N_{\mathcal E}}$ and $m^* = \mathbb{E}[\mathcal{E}_i]$ and $\Sigma^* = \mathbb{E}[(\mathcal E_i - m)(\mathcal E_i - m)^T]$. If
	\[ N_{\mathcal{E}} \ge \widehat C C_{\mathcal{E}}^4 \frac{L^2}{\eta^2} \log \left ( \frac{2L^2}{\alpha} \right ) \left [ \varepsilon^4 + h^{4(s+1)} \right ], \]
	where $s$ is given by Lemma \ref{fh_goes_to_0} and $\widehat C$ is specified in the proof, then
	\[ \mathbb{P} \left ( \norm{\Sigma - \Sigma^*}_2 \le \eta \right ) \ge 1 - \alpha. \]
\end{theorem}
\begin{proof}
	First, repeating verbatim the first part of the proof of Theorem \ref{inequality_modelling_error} we have
		\begin{equation}
		\abs{(\mathcal{E}_i)_l} \le \norm{\mathcal{E}_i}_2 \le K \varepsilon + \tilde{K} h^{s+1} \le C_{\mathcal{E}} ( \varepsilon + h^{s+1} ).
		\end{equation}
	Observe that if
	\begin{equation}
	\abs{\Sigma_{j,k} - \Sigma^*_{j,k}} \le \frac{\eta}{L} \qquad \text{for each } j,k = 1, \dots, L,
	\end{equation}
	then denoting by $\norm{\cdot}_F$ the Frobenius norm we have
	\begin{equation}
	\norm{\Sigma - \Sigma^*}_2 \le \norm{\Sigma - \Sigma^*}_F = \left ( \sum_{j,k=1}^L \abs{\Sigma_{j,k} - \Sigma^*_{j,k}}^2 \right )^{\frac{1}{2}} \le \eta,
	\end{equation}
	which implies that
	\begin{equation}
	\label{1toL_cov}
	\mathbb{P}( \norm{\Sigma - \Sigma^*}_2 \le \eta ) \ge \mathbb{P} \left ( \abs{\Sigma_{j,k} - \Sigma^*_{j,k}} \le \frac{\eta}{L} \quad \forall \; j,k = 1, \dots, L \right ).
	\end{equation}
	For all $j,k = 1, \dots, L$ define the functions $\phi_{j,k} \colon (\R^L)^N \to \R$ as
	\[ \phi_{j,k}(x_1, \dots, x_N) = \frac{1}{N-1} \sum_{i=1}^N (x_i^{(j)} - \bar{x}^{(j)}) (x_i^{(k)} - \bar{x}^{(k)}), \]
	where
	\[ \bar{x} = \frac{1}{N} \sum_{i=1}^N x_i, \]
	and the function $\Phi \colon (\R^L)^N \to \R^{L \times L}$ whose component $(j,k)$ is given by $\phi_{j,k}$. Observe that $\Sigma = \Phi(\mathcal{E}_1, \dots, \mathcal{E}_N) $ and $\Sigma^* = \mathbb{E}[\Phi(\mathcal{E}_1, \dots, \mathcal{E}_N)]$. Since the modelling error is bounded, we can restrict the functions $\phi_{j,k}$ to the ball of radius $C_{\mathcal E}(\varepsilon + h^{s+1})$ centred in $0$, $\phi_{j,k} \colon (B_{C_{\mathcal E}(\varepsilon + h^{s+1})})^N \to \R$, allowing us to prove the following bound
		\begin{equation}
		\label{ci}
		\abs{\phi_{j,k} - \phi_{j,k}'} \le \frac{48}{N} C_{\mathcal E}^2 (\varepsilon^2 + h^{2(s+1)}),
		\end{equation}
		where 
		\[ \phi_{j,k} = \phi_{j,k}(x_1, \dots, x_{i-1}, x_i, x_{i+1}, \dots, x_N) \qquad \text{and} \qquad \phi_{j,k}' = \phi_{j,k}(x_1, \dots, x_{i-1}, x_i', x_{i+1}, \dots, x_N). \]
		In fact we have
	\begin{equation}
	\label{decomposition_phi}
	\begin{aligned}
	\abs{\phi_{j,k} - \phi_{j,k}'} &= \left | \frac{1}{N-1} \sum_{n \neq i} (x_n^{(j)} - \bar{x}^{(j)}) (x_n^{(k)} - \bar{x}^{(k)}) + \frac{1}{N-1} (x_i^{(j)} - \bar{x}^{(j)}) (x_i^{(k)} - \bar{x}^{(k)}) \right. \\
	&\quad \left. - \frac{1}{N-1} \sum_{n \neq i} (x_n^{(j)} - \bar{x}'^{(j)}) (x_n^{(k)} - \bar{x}'^{(k)}) - \frac{1}{N-1} (x_i'^{(j)} - \bar{x}'^{(j)}) (x_i'^{(k)} - \bar{x}'^{(k)}) \right | \\
	&\le \frac{1}{N-1} \sum_{n \neq i} \left | (x_n^{(j)} - \bar{x}^{(j)}) (x_n^{(k)} - \bar{x}^{(k)}) - (x_n^{(j)} - \bar{x}'^{(j)}) (x_n^{(k)} - \bar{x}'^{(k)}) \right | \\
	&\quad + \frac{1}{N-1} \left | (x_i^{(j)} - \bar{x}^{(j)}) (x_i^{(k)} - \bar{x}^{(k)}) - (x_i'^{(j)} - \bar{x}'^{(j)}) (x_i'^{(k)} - \bar{x}'^{(k)}) \right |, \\
	&\eqdef Q_1 + Q_2,
	\end{aligned}
	\end{equation}
	where
	\begin{equation}
	\label{means}
	\bar{x} = \frac{1}{N} \left ( \sum_{n \neq i} x_n + x_i \right ) \qquad \text{and} \qquad \bar{x}' = \frac{1}{N} \left ( \sum_{n \neq i} x_n + x'_i \right ).
	\end{equation}
	Now we bound the two terms separately. First, we have
		\begin{equation}
		Q_1 \le \frac{1}{N-1} \abs{\bar{x}^{(k)} - \bar{x}'^{(k)}} \sum_{n \neq i} \abs{x_n^{(j)} - \bar{x}^{(j)}} + \frac{1}{N-1} \abs{\bar{x}^{(j)} - \bar{x}'^{(j)}} \sum_{n \neq i} \abs{x_n^{(k)} - \bar{x}'^{(k)}}.
		\end{equation}
		Let $x,y \in B_{C_{\mathcal E}(\varepsilon + h^{s+1})}$ and note that for all $j=1,\dots,L$ we have
		\begin{equation}
		\label{bound_component_cov}
		\abs{x^{(j)} - y^{(j)}} \le \norm{x - y}_2 \le 2 C_{\mathcal E} ( \varepsilon + h^{s+1} ).
		\end{equation}
		By equations \eqref{means} and \eqref{bound_component_cov}, it holds for all $j=1,\dots,L$
		\begin{equation*}
		\abs{\bar{x}^{(j)} - \bar{x}'^{(j)}} = \frac{1}{N} \left | \sum_{n \neq i} x_n^{(j)} + x_i^{(j)} - \sum_{n \neq i} x_n^{(j)} - x_i'^{(j)} \right | = \frac{1}{N} \abs{x_i^{(j)} - x_i'^{(j)}} \le \frac{2}{N} C_{\mathcal E} (\varepsilon + h^{s+1}),
		\end{equation*}
		therefore we obtain
		\begin{equation}
		\label{result1}
		Q_1 \le \frac{8}{N} C_{\mathcal E}^2 (\varepsilon + h^{s+1})^2 \le \frac{16}{N} C_{\mathcal E}^2 (\varepsilon^2 + h^{2(s+1)}).
		\end{equation}
		Moreover, by \eqref{bound_component_cov} we also get
		\begin{equation}
		\label{result2}
		Q_2 \le \frac{8}{N-1} C_{\mathcal E}^2 (\varepsilon + h^{s+1})^2 \le \frac{16}{N-1} C_{\mathcal E}^2 (\varepsilon^2 + h^{2(s+1)}),
		\end{equation}
		witch, together with \eqref{decomposition_phi}, \eqref{result1} and the fact that $(2N-1)/(N-1) \le 3$ for all $N\ge2$, implies \eqref{ci}
		\begin{equation}
		\abs{\phi_{j,k} - \phi_{j,k}'} \le \frac{16(2N-1)}{N(N-1)} C_{\mathcal E}^2 (\varepsilon^2 + h^{2(s+1)}) \le \frac{48}{N} C_{\mathcal E}^2 (\varepsilon^2 + h^{2(s+1)}).
		\end{equation}
	Therefore, applying McDiarmid's inequality we have
	\begin{equation}
	\label{mcdiarmid1}
	\mathbb{P} \left ( \abs{\Sigma_{j,k} - \Sigma^*_{j,k}} \ge \frac{\eta}{L} \right ) \le 2 \exp \left \{ - \frac{2 \eta^2 N_{\mathcal E}}{2304 L^2 C_{\mathcal E}^4 (\varepsilon^2 + h^{2(s+1)})^2} \right \} \le 2 \exp \left \{ - \frac{\eta^2 N_{\mathcal E} }{\widehat C L^2 C_{\mathcal E}^4 (\varepsilon^4 + h^{4(s+1)})} \right \},
	\end{equation}
	where $\widehat C = 2304$. Finally, we define the events $A_{j,k} = \left \{ \abs{\Sigma_{j,k} - \Sigma^*_{j,k}} \le \frac{t}{L} \right \}$ for each $j,k = 1, \dots, L$ and we repeat the same argument as in the last part of the proof of Theorem \ref{inequality_modelling_error}. Hence, due to \eqref{1toL_cov} and \eqref{mcdiarmid1} we have
		\begin{equation}
		\begin{aligned}
		\mathbb{P}( \norm{\Sigma - \Sigma^*}_2 \le \eta ) &\ge 1 - L^2 \max_{j,k=1,\dots,L} \mathbb{P} \left ( \abs{\Sigma_{j,k} - \Sigma^*_{j,k}} \ge \frac{\eta}{L} \right ) \\
		&\ge 1 - 2L^2 \exp \left \{ - \frac{\eta^2 N_{\mathcal E}}{\widehat C L^2 C_{\mathcal E}^4 (\varepsilon^4 + h^{4(s+1)})} \right \},
		\end{aligned}
		\end{equation}
		and if $N_{\mathcal{E}}$ satisfies the hypothesis we obtain the desired result.
\end{proof}

\begin{remark}
	Note that, in Theorem \ref{inequality_modelling_error} and Theorem \ref{inequality2_modelling_error}, as expected, the number $N_{\mathcal{E}}$ of full multiscale problems tends to infinity if we require no error between the sample and the true mean and covariance ($\eta \to 0$) or certainty that the error is below a certain value ($\alpha \to 0$). Moreover, observe that for any given accuracy the number of samples required $N_{\mathcal E}$ is a increasing function of $\epl$ and $h$, so that if the model $\mathcal G^0_h$ is a good approximation of $\mathcal G$, thus computationally expensive, then only few samples are needed. In particular, notice that in order to obtain a good approximation of the true mean, the number of full multiscale problems is
		\begin{equation}
		N_{\mathcal{E}} = \mathcal O \left( \eta^{-2} \log(\alpha^{-1}) \left( \epl^2 + h^{2(s+1)} \right) \right),
		\end{equation}
		while to have a reliable approximation of the covariance matrix it is required that
		\begin{equation}
		N_{\mathcal{E}} = \mathcal O \left( \eta^{-2} \log(\alpha^{-1}) \left( \epl^4 + h^{4(s+1)} \right) \right).
		\end{equation}
\end{remark}

\section{Numerical experiments}\label{Experiments}

In this section, using the setting of \cite{AbD20}, we present some numerical experiments to illustrate the iterative ensemble Kalman method to solve multiscale inverse problems. \\
Let $\Omega$ be a bounded open domain. We consider a class of parametrized multiscale locally periodic tensors of the type $A^{\varepsilon}_{\sigma^*}(x) = A(\sigma^*(x),x/\varepsilon)$, where $\sigma^* \colon \Omega \to \R$. We assume to know the map $(t,x) \to A(t,x/\varepsilon)$ for all $x \in \Omega$ and $t \in \R$ and we want to estimate the function $\sigma^*$ given measurements computed from the model
\begin{equation}
\label{pb}
\begin{cases}
- \nabla \cdot ( A^{\varepsilon}_{\sigma^*} \nabla p^{\varepsilon} ) = 0 & \text{ in } \Omega, \\
p^{\varepsilon} = g & \text{ on } \partial \Omega.
\end{cases}
\end{equation}
\begin{remark}
	Note that the theory has been developed for Dirichlet homogeneous boundary conditions, but it can be applied to the non-homogeneous case by considering an extension of the function at the boundary and slightly modifying the PDE. For more details we refer to \cite[Remark 8.10]{Sal16}.
\end{remark}
For the unknown $\sigma^*$ we consider the following admissible set
\[ \Sigma = \{ \sigma \in L^{\infty}(\Omega) \colon \sigma^- \le \sigma(x) \le \sigma^+ \}, \]
where $\sigma^-$ and $\sigma^+$ are two given values. \\
The measurements, which we take into account, are the integrals of the normal flux multiplied by some functions with compact support in a portion of the boundary of the domain. More precisely, we consider $I \in \N$ disjoint portions of $\Omega$, which we denote by $\Gamma_i \in \partial \Omega$, $i = 1, \dots, I$, $\Gamma_i \cap \Gamma_j = \emptyset$ for $i \neq j$, and $I$ functions $\phi_i \in H^{1/2}(\partial \Omega)$ with compact support $\mathrm{supp} \; (\phi_i) \subset \Gamma_i$ for all $i = 1, \dots, I$. Moreover, we solve \eqref{pb} for $K \in \N$ Dirichlet data $g_k$, $k = 1, \dots, K$, and we denote by $p_k^\epl$ the solution of the problem. Let $\Lambda_{A^\epl_\sigma} \colon H^{1/2}(\partial \Omega) \to H^{-1/2}(\partial \Omega)$ be the operator which maps the Dirichlet data $g$ to the normal flux of the solution $p^\epl$ of \eqref{pb} 
	\begin{equation}
	\Lambda_{A^\epl_\sigma} g = A^\epl_\sigma \nabla p^\epl \cdot \nu,
	\end{equation}
	where $\nu$ is the exterior unit normal vector to $\partial \Omega$.
	Then we define the multiscale operator $\mathcal{F}^{\varepsilon} \colon \Sigma \to \R^L$ where $L = IK$ by components
	\begin{equation}
	\label{boundary_integral_e}
	\mathcal{F}^{\varepsilon}(\sigma)_{ik} = \mathcal{F}^{\varepsilon}(\sigma)_{l} = \left\langle \Lambda_{A^\epl_\sigma} g_k , \varphi_i \right\rangle_{H^{-1/2}(\partial \Omega), H^{1/2}(\partial \Omega)}, \qquad i = 1, \dots, I, \quad k = 1, \dots, K,
	\end{equation}
	which, with an abuse of notation, can be written
	\begin{equation}
	\mathcal{F}^{\varepsilon}(\sigma)_{ik} = \int_{\Gamma_i} A^{\varepsilon} \nabla p_k^{\varepsilon} \cdot \nu \phi_i ds.
	\end{equation}
The final vector of observations $y$ is given by the sum of the operator $\mathcal{F}^{\varepsilon}$ and a noise
\[ y = \mathcal{F}^{\varepsilon}(\sigma^*) + \eta, \]
where $\eta \sim \mathcal{N}(0, \Gamma)$ and $\Gamma$ is a given symmetric positive definite covariance matrix, which, in our experiments, is a multiple of the identity $\Gamma = \gamma^2 I$ and $\gamma$ is a given value. Observations are computed with a refined Finite Element Method (FEM) with mesh size $h_{\mathrm{obs}} \ll \varepsilon$, while the homogenized version of problem \eqref{pb} is solved using a macro mesh size $h \gg h_{\mathrm{obs}}$. We call $\mathcal{T}_h$ the macro triangulation and $N_h$ the total number of nodes defining $\mathcal{T}_h$. We assume that the prior distribution for the discretization of the unknown $\sigma^*$ on the macro triangulation $\mathcal{T}_h$ is given by $\mathcal{N}(\sigma_0, C)$, where $\sigma_0$ is a given discretization of a function in $\Sigma$ and $C \in \R^{N_h \times N_h}$ is defined by
\begin{equation*}
C_{ij} = \delta \exp \left ( - \frac{\norm{x_i - x_j}_2}{\lambda} \right ),
\end{equation*}
where $\delta, \lambda \in \R^+$ and $\{ x_i \}_{i=1}^{N_h}$ are the nodes of the macro triangulation $\mathcal{T}_h$. The parameter $\lambda$ is a correlation length that describes how the values at different positions of the functions supported by the prior measure are related, while the parameter $\delta$ is an amplitude scaling factor. Regarding the prior modelling, we need to take into account that even if in the homogenized problem the coarse and fine scales have been separated, functions drawn from the prior distribution on the coarse scale can exhibit multiple scales, including the fine scale of our multiscale model, depending on the rate of decay of the prior covariance. This issue can thus be controlled by setting the parameters $\delta$ and $\lambda$. Even though this does not ensure a clear separation between coarse and fine scales, our numerical results illustrate that it is sufficient in practice. \\
In order to reduce the dimensionality of the unknown we use a truncated Karhunen-Lo\`eve expansion. Any sample from the prior distribution $\mathcal{N}(\sigma_0,C)$ can be represented as
\begin{equation}
\label{KL}
\sigma = \sigma_0 + \sum_{m=1}^{N_h} \sqrt{\lambda_m} u_m \psi_m,
\end{equation}
where $\{ \psi_m \}_{m=1}^{N_h}$ is an orthonormal set of eigenvectors of $C$ with corresponding eigenvalues $\{ \lambda_m \}_{m=1}^{N_h}$ in decreasing order, and $\{ u_m \}_{m=1}^{N_h}$ is an i.i.d sequence with $u_m \sim \mathcal{N}(0,1)$. Note that the Karhunen-Lo\`eve expansion works also in the infinite dimensional setting, where $\sigma_0 \in \Sigma$, $C$ is a covariance operator and $\{ \lambda_m, \psi_m \}_{m=1}^{\infty}$ is an orthonormal set of eigenvalues-eigenfunctions with respect to the scalar product in $L^2(\Omega)$. Then the truncated Karhunen-Lo\`eve expansion of the discretization of $\sigma$ consists of taking the first $M$ components of the series in \eqref{KL}
\begin{equation}
\label{KL_truncated}
\sigma \simeq \sigma_0 + \sum_{m=1}^{M} \sqrt{\lambda_m} u_m \psi_m,
\end{equation}
and the actual unknown becomes the vector $u \in \R^M$, whose components are the coefficients $u_m$ in \eqref{KL_truncated}. Then we define the multiscale forward operator $\mathcal{G}^{\varepsilon} \colon \R^M \to \R^L$ as the composition of $\mathcal{F}^{\varepsilon}$ with the truncated Karhunen-Lo\`eve expansion
\[ \mathcal{G}^{\varepsilon}(u) = \mathcal{F}^{\varepsilon} \left ( \sigma_0 + \sum_{m=1}^{M} \sqrt{\lambda_m} u_m \psi_m \right ). \]
In the iterative ensemble Kalman method we do not compute the exact solution of problem \eqref{pb}, but we solve its homogenized version numerically using the macro triangulation $\mathcal{T}_h$, therefore we obtain the homogenized discrete solution $p^0_h$. The problem is solved applying the finite element heterogeneous multiscale method (FE-HMM), which is described in \cite{Abd11, AEE12}. Hence, analogously to the multiscale case, we define the discrete homogenized operator $\mathcal{F}_h^0 \colon \Sigma \to \R^L$ with an abuse of notation as
\begin{equation}
\label{boundary_integral_0}
\mathcal{F}_h^{0}(\sigma)_{l} = \mathcal{F}_h^{0}(\sigma)_{ik} = \int_{\Gamma_i} A^{0} \nabla p_{h_k}^{0} \cdot \nu \phi_i ds, \qquad i = 1, \dots, I, \quad k = 1, \dots, K,
\end{equation}
and the discrete homogenized forward operator $\mathcal{G}^0_h \colon \R^M \to \R^L$, which is actually used in the algorithm, as
\[ \mathcal{G}_h^{0}(u) = \mathcal{F}_h^{0} \left ( \sigma_0 + \sum_{m=1}^{M} \sqrt{\lambda_m} u_m \psi_m \right ). \]
Finally, we call $u_{\mathrm{EnKF}}$ the solution of the iterative ensemble Kalman algorithm and the estimated $\sigma_{\mathrm{EnKF}}$ is obtained from the truncated Karhunen-Lo\`eve expansion
\[ \sigma_{\mathrm{EnKF}} = \sigma_0 + \sum_{m=1}^M \sqrt{\lambda_m} u_{\mathrm{EnKF}_m} \psi_m. \]

\subsection{Data}

In the numerical results presented in the following section the computational domain is the unit square
\[ \Omega = (0,1)^2 \subset \R^2. \]
For the discretization parameters we set $\varepsilon = 1/64$ and $h_{\mathrm{obs}} = 1/4096$ and for the forward homogenized problem we use a macro mesh size $h = 1/32$, which is much larger than $h_{\mathrm{obs}}$ and reduces the computational cost significantly. We solve the problem for $K = 3$ Dirichlet conditions $\{ g_k \}_{k=1}^3$ and $g_k = \sqrt{\mu_k} \theta_k$ where $\{ (\mu_k, \theta_k) \}_{k=1}^3$ are couples of eigenvalues and eigenfunctions of the one dimensional discrete Laplacian operator corresponding to the first $K = 3$ smallest eigenvalues. For each $g_k$ we consider its restriction to the boundary $\partial \Omega$ in order to obtain a Dirichlet condition. These functions are orthonormal with respect to the scalar product in $L^2(\Omega)$ and this ensures that each function gives independent information. \\
To compute the boundary integrals in \eqref{boundary_integral_e} and \eqref{boundary_integral_0}, we consider $I = 12$ boundary portions, three for each side of the square $\Omega$. In particular, for each side, all $\Gamma_i$ have length equal to $0.2$ and they consist of the intervals $(0.1,0.3), (0.4,0.6)$ and $(0.7,0.9)$. The functions $\{ \phi_i \}_{i=1}^{12}$ are hat functions with $\mathrm{supp} \; (\phi_i) = \Gamma_i$, which take value one at the midpoint and value $0$ at the extremes of $\Gamma_i$. Then the parameter of the noise, which perturbs the observations, is $\gamma = 0.01$. \\
Moreover, regarding the prior distribution for the unknown, we consider $\sigma_0 = 0$ and the parameters of the covariance matrices are $\delta = 0.05$ and $\lambda = 0.5$. In the truncated Karhunen-Lo\`eve expansion we take $M = 100$. Finally, about the ensemble Kalman method, we consider $J = 1000$ particles for each ensemble and $500$ iterations. \\
The exact tensor $A_{\sigma^*}^{\varepsilon}$ is given by
\begin{align*}
a_{11} \left ( \sigma^*(x), \frac{x}{\varepsilon} \right ) = & \; e^{\sigma^*(x)} \left ( \cos^2 \left ( \frac{2 \pi x_1}{\varepsilon} \right ) + 1 \right ) + \cos^2 \left ( 2 \pi \frac{x_2}{\varepsilon} \right ), \\
a_{12} \left ( \sigma^*(x), \frac{x}{\varepsilon} \right ) = & \; 0, \\
a_{21} \left ( \sigma^*(x), \frac{x}{\varepsilon} \right ) = & \; 0, \\
a_{22} \left ( \sigma^*(x), \frac{x}{\varepsilon} \right ) = & \; e^{\sigma^*(x)} \left ( \sin \left ( \frac{2 \pi x_2}{\varepsilon} \right ) + 2 \right ) + \cos^2 \left ( 2 \pi \frac{x_1}{\varepsilon} \right ),
\end{align*}
where
\begin{equation*}
\sigma^*(x) = \log (1.3 + 0.3 \mathbbm{1}_{D_1} - 0.4 \mathbbm{1}_{D_2}),
\end{equation*}
and
\begin{align*}
D_1 = & \left \{ x = (x_1,x_2) \colon \left ( x_1 - \frac{5}{16} \right )^2 + \left ( x_2 - \frac{11}{16} \right )^2 \le 0.025 \right \}, \\
D_2 = & \left \{ x = (x_1,x_2) \colon \left ( x_1 - \frac{11}{16} \right )^2 + \left ( x_2 - \frac{5}{16} \right )^2 \le 0.025 \right \}.
\end{align*}
Figure \ref{fig:exact_unknown} shows the exact unknown $\sigma^*$. Note that $\sigma^*$ is a non-continuous function, but, in order to approximate it, we are using a truncated Karhunen-Lo\`eve expansion, where the eigenfunctions are smooth.

\begin{figure}[t]
	\centering
	\includegraphics[]{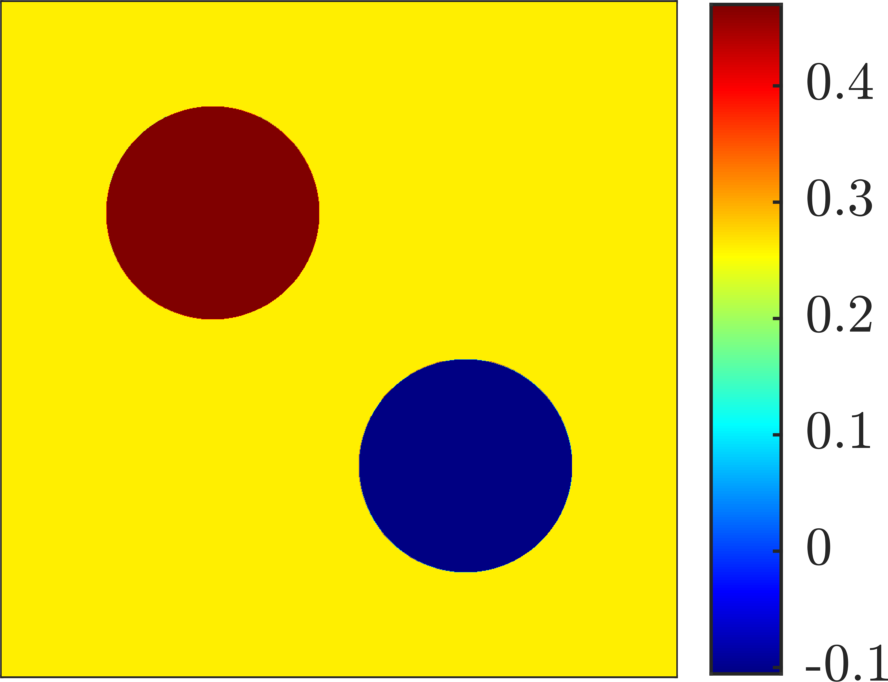}
	\caption{Exact unknown $\sigma^*$ employed for numerical experiments.}
	\label{fig:exact_unknown}
\end{figure}

One can verify that the tensor $A^{\varepsilon}_{\sigma}$ satisfies Assumption \ref{ass_A}. In particular, for $\xi \in \R^2$ we have
\begin{equation*}
A^{\varepsilon}_{\sigma} \xi \cdot \xi = a_{1,1} \left ( \sigma(x), \frac{x}{\varepsilon} \right ) \xi_1^2 + a_{2,2} \left ( \sigma(x), \frac{x}{\varepsilon} \right ) \xi_2^2 \ge e^{\sigma(x)} ( \xi_1^2 + \xi_2^2 ) \ge e^{\sigma_-} \norm{\xi}_2^2.
\end{equation*}
Moreover, since the $\mathrm{EnKF}$ algorithm estimates the coefficients $\{ u_m \}_{m=1}^M$ of the truncated Karhunen-Lo\`eve expansion, we show that $A^{\varepsilon}(u) \colon \R^M \to L^{\infty}(\Omega, \R^{d \times d})$, which maps $u$ into $A^{\varepsilon}_{\sigma_u}$, is Lipschitz. In fact we first have
	\begin{equation*}
	\norm{A^{\varepsilon}(u_1) - A^{\varepsilon}(u_2)}_{L^{\infty}(\Omega, \R^{d \times d})} \le \sqrt{13} e^{\sigma^+} \sup_{x \in \Omega} \abs{\sigma_{u_1}(x) - \sigma_{u_2}(x)},
	\end{equation*}
	then using the truncated Karhunen-Lo\`eve expansion and the Cauchy-Schwarz inequality we obtain
	\begin{equation*}
	\norm{A^{\varepsilon}(u_1) - A^{\varepsilon}(u_2)}_{L^{\infty}(\Omega, \R^{d \times d})} \le \sqrt{13} e^{\sigma^+} \sup_{x \in \Omega} \left ( \sum_{m=1}^M \lambda_m \psi_m^2(x) \right )^{1/2} \norm{u_1 - u_2}_2,
	\end{equation*}
	which shows that $A^\epl(u)$ is Lipschitz with constant equal to $\sqrt{13} e^{\sigma^+} \sup_{x \in \Omega} \left ( \sum_{m=1}^M \lambda_m \psi_m^2(x) \right )^{1/2}$.

\subsection{Results}

\begin{figure}[t]
	\centering
	\begin{tabular}{cccc}
		\includegraphics[]{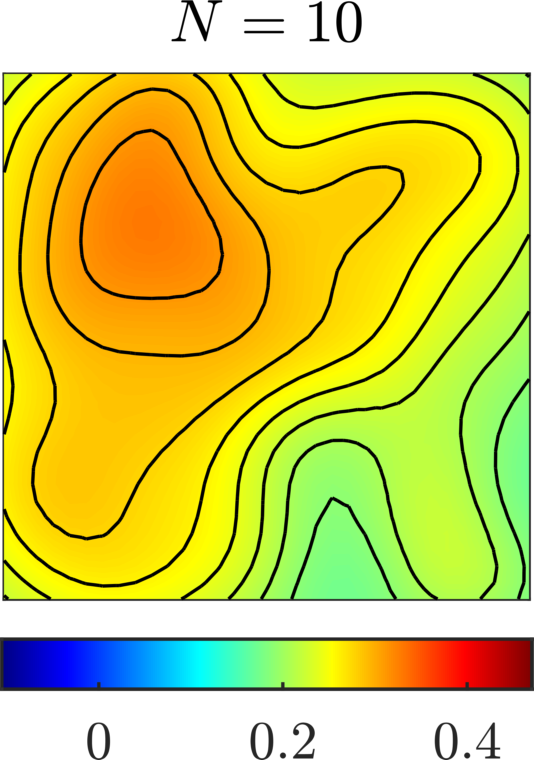} &
		\includegraphics[]{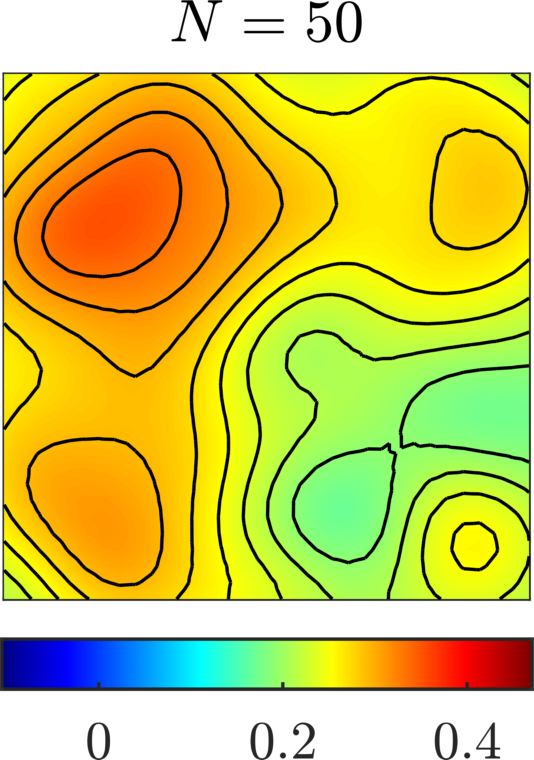} &
		\includegraphics[]{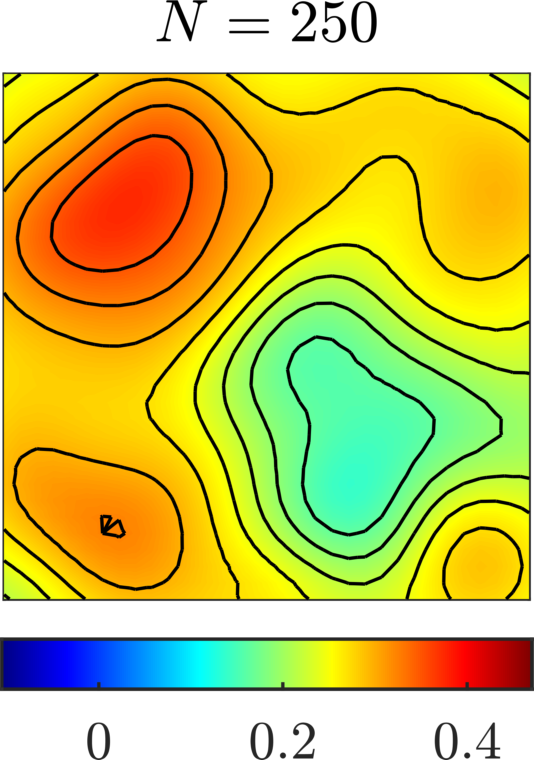} &
		\includegraphics[]{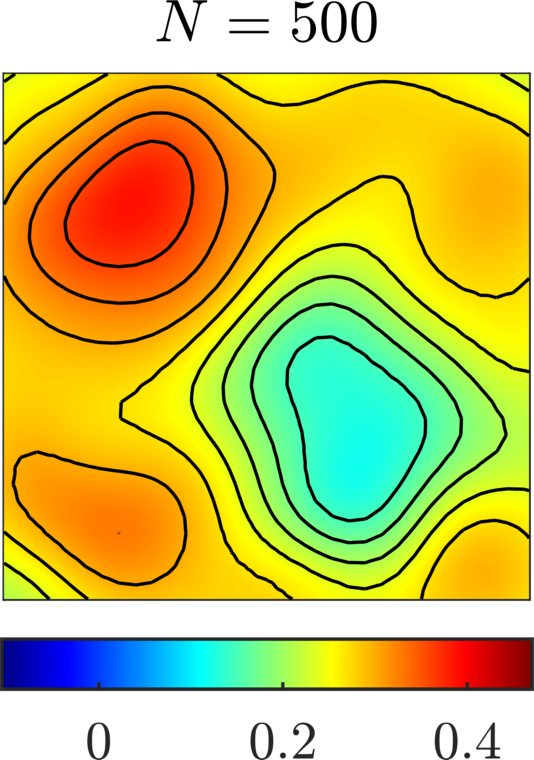}
	\end{tabular}
	\caption{EnKF estimation after $N = \{10, 50, 250, 500\}$ iterations.}
	\label{fig:best_solution}
\end{figure}

We first fix the multiscale parameter $\epl = 1/32$ and the ensemble size $J = 500$ and study the evolution with respect to the number of steps. In Figure \ref{fig:best_solution} we plot the estimation $\sigma_{\mathrm{EnKF}}$ after $10, 50, 250$ and $500$ iterations of the ensemble Kalman algorithm. We clearly see that the approximation gets better as the number of iterations increases and that convergence has been reached. In particular, already after $N = 250$ iterations the algorithm seem to have reached convergence. We point out that we obtain a quite good approximation of the real unknown $\sigma^*$ , indeed we are trying to recover a non-continuous function in the whole domain given only some observations at the boundary.

\begin{figure}[t]
	\centering
	\begin{tabular}{cccc}
		\includegraphics[]{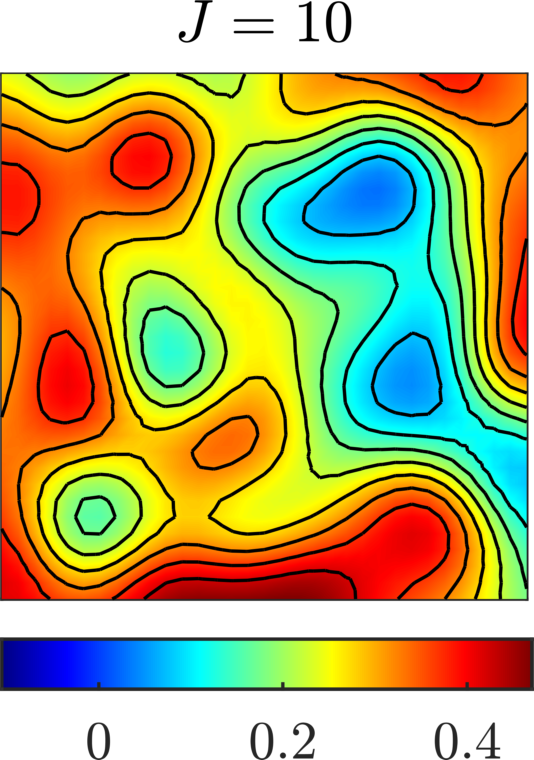} &
		\includegraphics[]{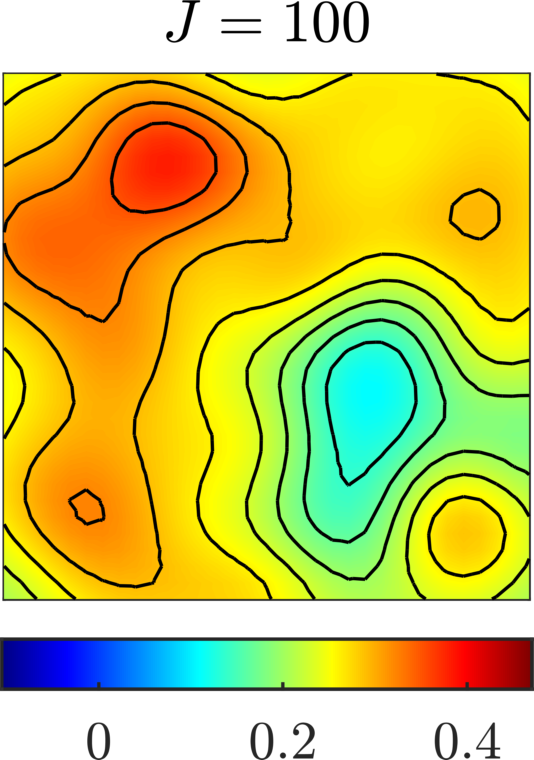} &
		\includegraphics[]{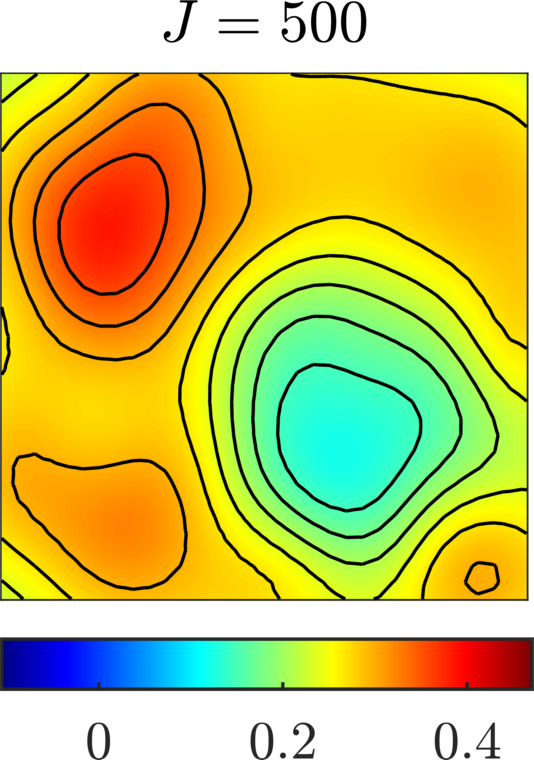} &
		\includegraphics[]{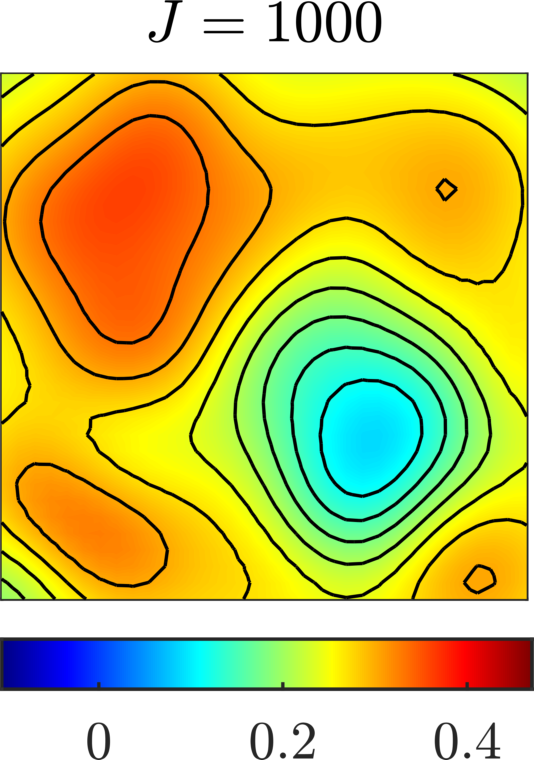}
	\end{tabular}
	\caption{EnKF estimation after $N = 500$ iterations with ensemble size $J = \{10, 100, 500, 1000\}$.}
	\label{fig:comparison_J}
\end{figure}

We now perform a sensitivity analysis with respect to the ensemble size. In Figure \ref{fig:comparison_J} we vary the number of particles $J$ and we compare the results obtained at the end of the algorithm after $500$ iterations for $\epl=1/32$. As expected, the approximation becomes better when the ensemble contains more particles. In particular, note that if the number of particles is too small, e.g. $J = 10$, then the approximation is not satisfying. 

Further, we fix the ensemble size $J = 500$ and we perform $N = 500$ iterations of the EnKF for different values of the multiscale parameter. Results, shown in Figure \ref{fig:comparison_e}, highlight how the approximation becomes worse when $\varepsilon$ is bigger, indeed the homogenized problem becomes too different with respect to the multiscale one and, if $\varepsilon$ is too big, the solution does not approximate the true unknown.

\begin{figure}[t]
	\centering
	\begin{tabular}{cccc}
		\includegraphics[]{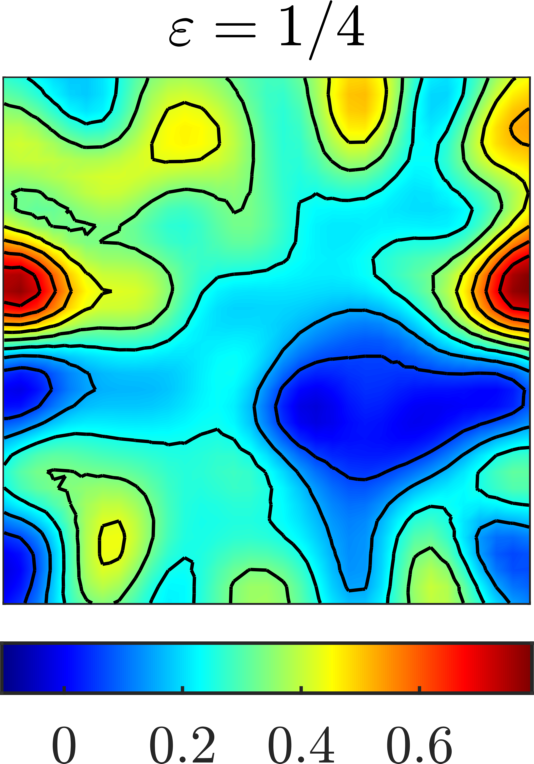} &
		\includegraphics[]{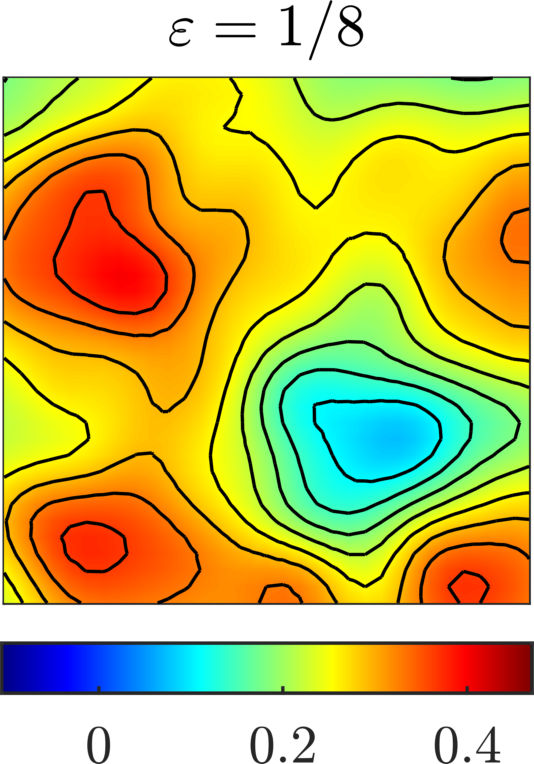} &
		\includegraphics[]{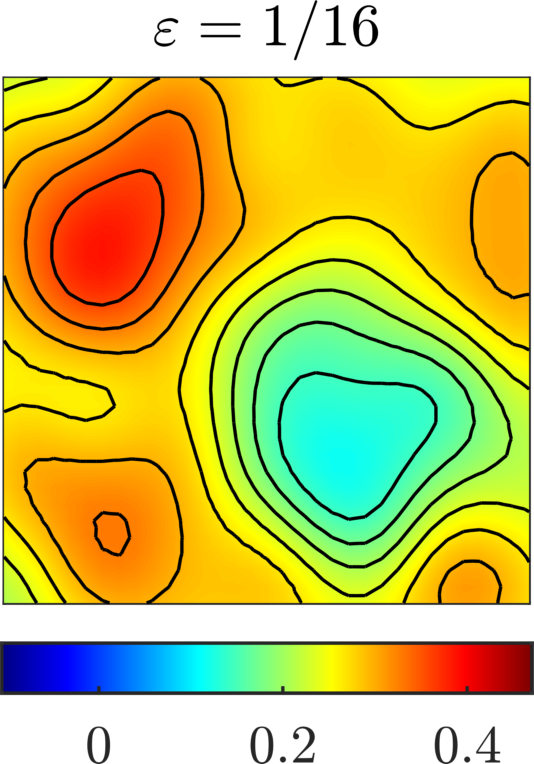} &
		\includegraphics[]{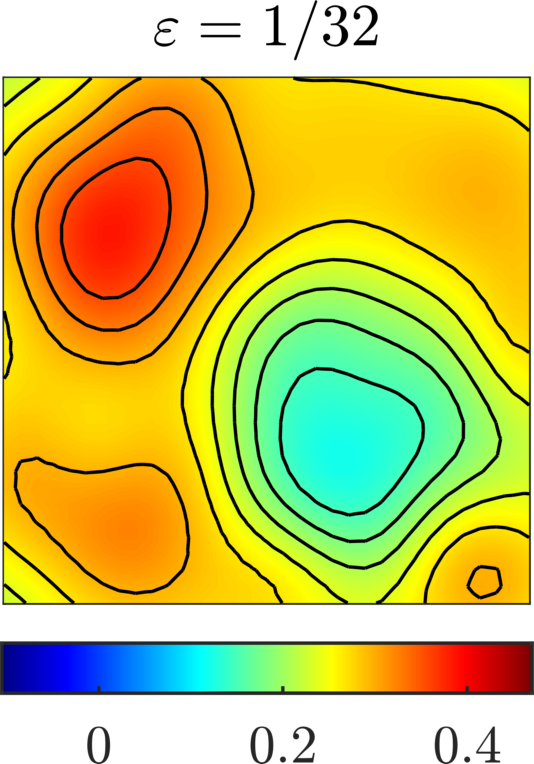}
	\end{tabular}
	\caption{EnKF estimation after $N = 500$ iterations for the multiscale parameter $\varepsilon = \{1/4, 1/8, 1/16, 1/32\}$.}
	\label{fig:comparison_e}
\end{figure}

Moreover, in order to obtain good results even in case $\varepsilon$ is not close to the asymptotic limit $\varepsilon \to 0$, in Figure \ref{fig:comparison_e_model_error} we apply offline modelling error estimation with $N_{\mathcal{E}} = 20$ and we plot the solution of the inverse problem \eqref{model_error_y4} for different values of the multiscale parameter $\varepsilon$. Comparing these plots with the ones in Figure \ref{fig:comparison_e}, in particular for $\varepsilon = 1/4$, we observe that the modelling error estimation significantly improves the results.

\begin{figure}[t]
	\centering
	\begin{tabular}{cccc}
		\includegraphics[]{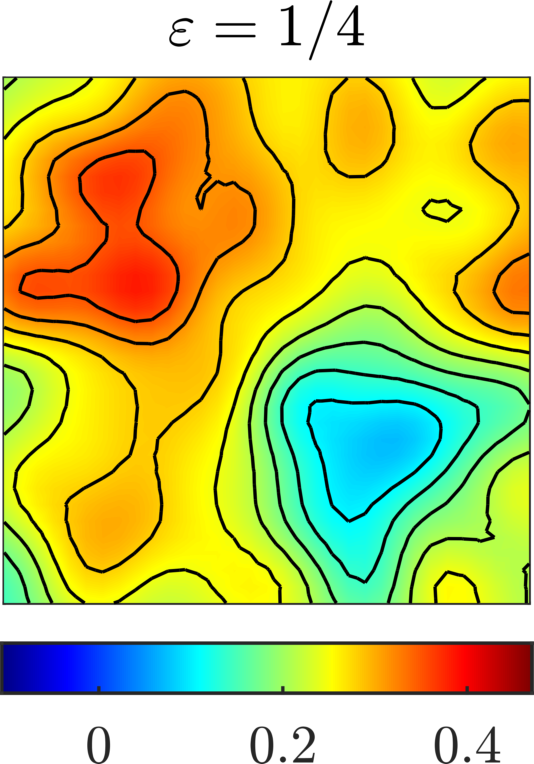} &
		\includegraphics[]{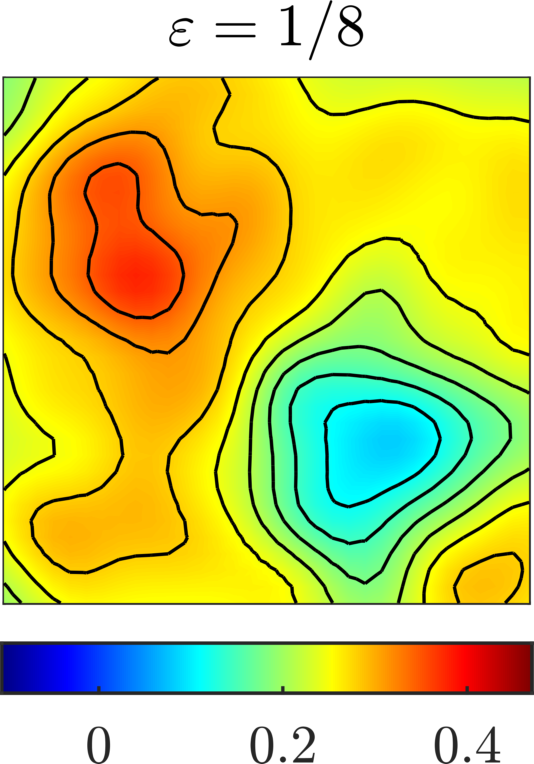} &
		\includegraphics[]{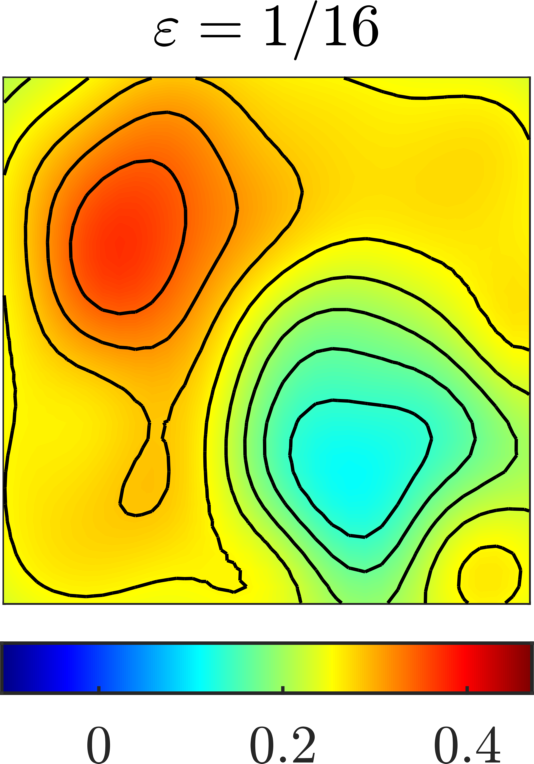} &
		\includegraphics[]{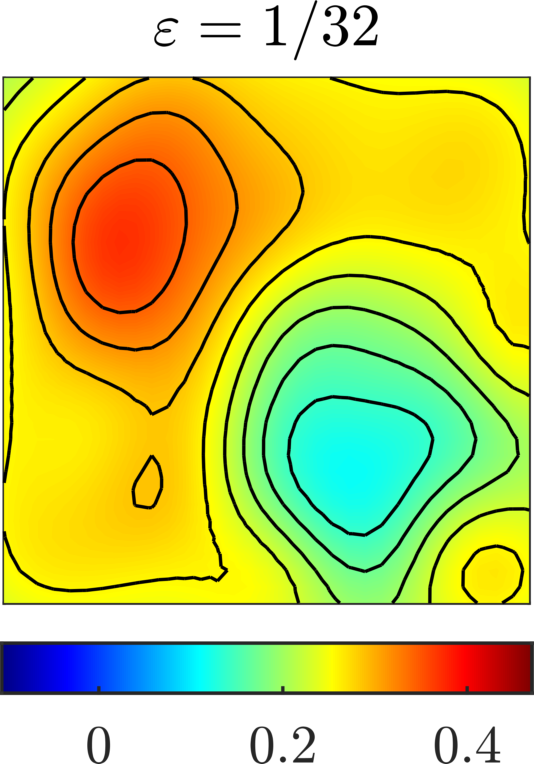}
	\end{tabular}
	\caption{EnKF with offline modeling error estimation after $500$ iterations for the multiscale parameter $\varepsilon = \{1/4, 1/8, 1/16, 1/32\}$.}
	\label{fig:comparison_e_model_error}
\end{figure}

Finally, in Figure \ref{fig:comparison_e_model_error_levels} we show the results obtained by applying the ensemble Kalman method with dynamic updating of the modelling error distribution with $\mathcal{L} = 5$ levels, $N_{\mathcal{E}}^{\ell} = 4$ samples and $N^{\ell} = 100$ iterations at each level $\ell = 1, \dots, \mathcal{L}$. The number of resolutions of the full multiscale problem is $20$ and the total number of iterations is $500$, which are equal to the previous approach, where the distribution of the modelling error was approximated offline. Comparing these plots with the ones in Figure \ref{fig:comparison_e_model_error}, we note that updating the distribution of the modelling error dynamically still improves the results.

\begin{figure}[t!]
	\centering
	\begin{tabular}{cccc}
		\includegraphics[]{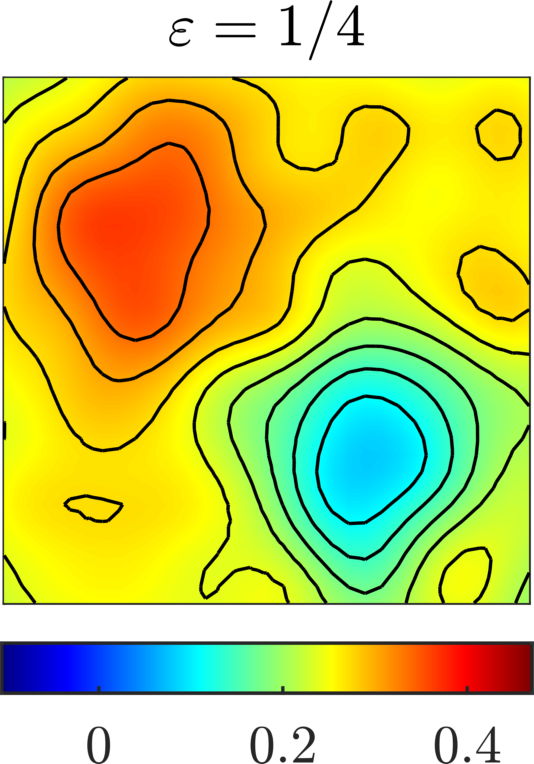} &
		\includegraphics[]{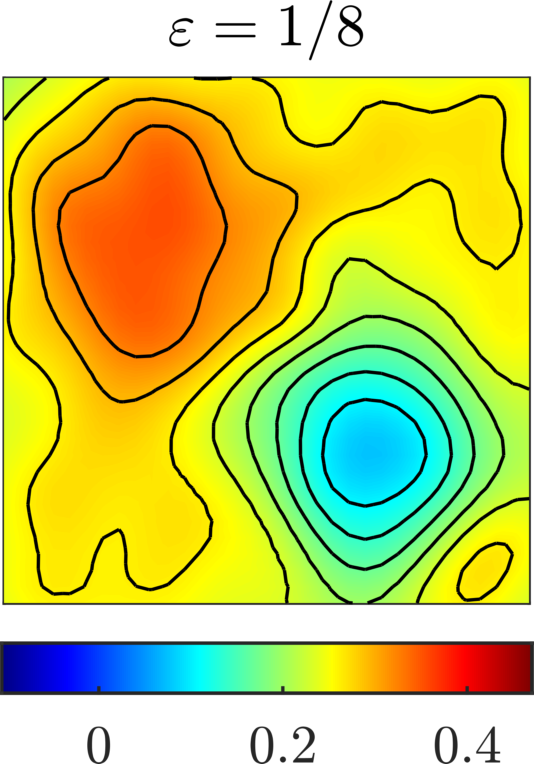} &
		\includegraphics[]{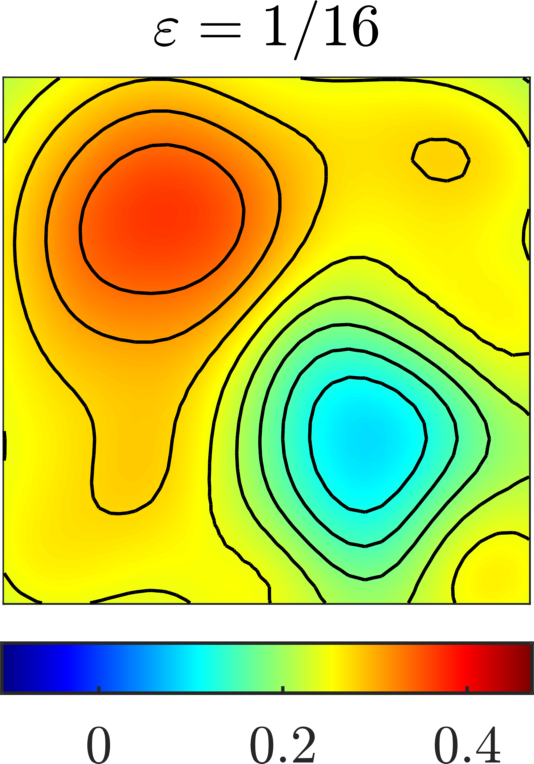} &
		\includegraphics[]{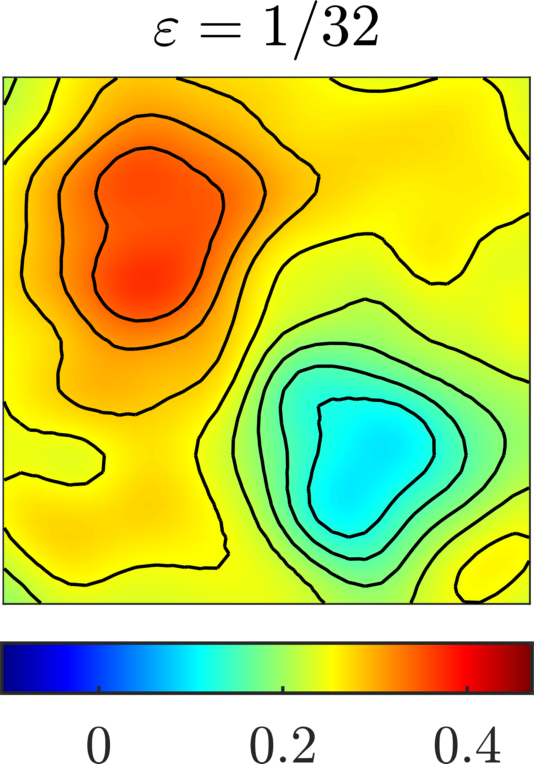}
	\end{tabular}
	\caption{EnKF with online iterative modeling error estimation after $500$ iterations for the multiscale parameter $\varepsilon = \{1/4, 1/8, 1/16, 1/32\}$.}
	\label{fig:comparison_e_model_error_levels}
\end{figure}

\section{Conclusion}

In this paper we analyzed the ensemble Kalman inversion methodology in the context of inverse problems for multiscale elliptic PDEs with tensors highly oscillatory at a scale $\epl \ll 1$. The multiscale algorithm we propose relies on the EnKF, on a surrogate homogenized forward operator and on numerical homogenization techniques such as the FE-HMM. It guarantees a significant reduction in computational cost for problems which would be otherwise computationally  involved or unfeasible. In  Theorem \ref{convergence_result_full} we have shown that the ensemble of particles approximating the unknown parameter generated by our multiscale algorithm converges to the ensemble generated by the true model as the small scale parameter $\epl$ and the numerical discretization parameter $h$ go to zero. Furthermore in a Bayesian framework, we have shown in Theorem \ref{convergence_posterior_distributions} that the discrete probability measure based on the ensemble originating from our multiscale algorithm converges to the measure generated by the true model, again as $\epl$ and $h$ go to zero.
Hence when $\epl \ll 1$ and the full model is expensive to solve, the multiscale numerical method we propose is both accurate and efficient to recover an unknown parameter in multiscale elliptic PDEs. Moreover, we equipped our method with a technique which allows to account for the discrepancy between the artificial homogenized surrogate forward model and the true multiscale data, thus alleviating the effects of model misspecification. This technique requires additional offline or online computations involving the numerical solution of the full multiscale problem. The optimal number of such additional solves is quantified in Theorem  \ref{inequality_modelling_error} and Theorem \ref{inequality2_modelling_error}. In particular, we have proved that the number of solves needed to reach any required accuracy tends to zero when the small scale parameter $\epl$ and the numerical discretization parameter $h$ vanish. Hence, we can conclude that accounting for model misspecification is particularly beneficial for mid-range values of $\epl$, when a small number of full solves should be computationally affordable. The efficiency and usefulness of the multiscale algorithm have been further demonstrated through a series of numerical experiments.

\clearpage 

\section*{Appendix}

\subsection*{Proof of Lemma \ref{G_lipschitz}}

Let $u_1, u_2 \in \mathbb{R}^M$, and $p_1 = \mathcal S(u_1)$, $p_2 = \mathcal S(u_2)$. From the weak formulations of \eqref{problem_lemma} we get that
\begin{equation*}
\int_\Omega \big(A_{u_1}\nabla p_1 - A_{u_2} \nabla p_2\big) \cdot \nabla v = 0 \qquad \text{for all } v \in H^1_0(\Omega),	
\end{equation*}
which yields
\begin{equation*}
\int_{\Omega} A_{u_1} (\nabla p_1 - \nabla p_2) \cdot \nabla v = - \int_{\Omega} (A_{u_1} - A_{u_2}) \nabla p_2 \cdot \nabla v.
\end{equation*}
Then choosing $v = p_1 - p_2$, by the hypotheses on $A_u$ and applying the H\"older inequality we obtain
\begin{equation*}
\alpha \norm{\nabla p_1 - \nabla p_2}_{L^2(\Omega; \R^d)}^2 \le M \norm{u_1 - u_2}_2 \norm{\nabla p_2}_{L^2(\Omega; \R^d)} \norm{\nabla p_1 - \nabla p_2}_{L^2(\Omega; \R^d)},
\end{equation*}
which due a standard coercivity argument implies 
\begin{equation} \label{intermediate}
\norm{\nabla p_1 - \nabla p_2}_{L^2(\Omega; \R^d)} \le \frac{M C_p}{\alpha^2} \norm{f}_{L^2(\Omega)} \norm{u_1 - u_2}_2,
\end{equation}
where $C_p$ is the Poincaré constant associated to the domain $\Omega$. Hence \eqref{intermediate} shows that $\mathcal{S}$ is Lipschitz with constant
\[ L_{\mathcal{S}} = \frac{M C_p}{\alpha^2} \norm{f}_{L^2(\Omega)}. \]
Finally, since $\mathcal{G}$ is the composition of two Lipschitz operators, we deduce that it is also Lipschitz with constant $L_{\mathcal G} = L_{\mathcal O} L_{\mathcal S}$.
\qed

\subsection*{Proof of Lemma \ref{f_goes_to_0}}

Let us consider an ensemble $u \in \mathcal U_{J,M}$ with particles $u^{(j)} \in \R^M$, for $j = 1, \ldots, J$. For each particle we have
\begin{equation}
\norm{\mathcal{G}^{\varepsilon}(u^{(j)}) - \mathcal{G}^0(u^{(j)})}_2 = \norm{\mathcal{O}(\mathcal{S}^{\varepsilon}(u^{(j)})) - \mathcal{O}(\mathcal{S}^0(u^{(j)}))}_2 \le C_{\mathcal O} \norm{p^{\varepsilon}(u^{(j)}) - p^0(u^{(j)})}_{L^2(\Omega)},
\end{equation}
where we write explicitly the dependence of the solutions $p^{\varepsilon}$ and $p^0$ on the particle they are generated by. Due to homogenization theory, we have that $p^{\varepsilon}(u^{(j)}) \toweak p^0(u^{(j)})$ in $H^1_0(\Omega)$ for all $j=1,\dots,J$, and therefore $p^{\varepsilon}(u^{(j)}) \to p^0(u^{(j)})$ in $L^2(\Omega)$, which implies
\[ e(\varepsilon, u) = \frac{1}{J} \sum_{j=1}^J \norm{\mathcal{G}^{\varepsilon}(u^{(j)}) - \mathcal{G}^0(u^{(j)})}_2 \le \frac{C_{\mathcal O}}{J} \sum_{j=1}^J \norm{p^{\varepsilon}(u^{(j)}) - p^0(u^{(j)})}_{L^2(\Omega)} \to 0. \]
Moreover, if the solution of the homogenized problem $p^0$ is sufficiently smooth independently of $u$, namely $p^0 \in H^2(\Omega)$, letting $C > 0$ be a constant independent of $\epl$, we have by \cite{MoV97} for all $j = 1, \dots, J$
\[ \norm{p^{\varepsilon}(u^{(j)}) - p^0(u^{(j)})}_{L^2(\Omega)} \le C \varepsilon, \]
which implies
\[ e(\varepsilon, u) = \frac{1}{J} \sum_{j=1}^J \norm{\mathcal{G}^{\varepsilon}(u^{(j)}) - \mathcal{G}^0(u^{(j)})}_2 \le \frac{C_{\mathcal O}}{J} \sum_{j=1}^J \norm{p^{\varepsilon}(u^{(j)}) - p^0(u^{(j)})}_{L^2(\Omega)} \le C_{\mathcal O} C \varepsilon, \]
and defining $K = C_{\mathcal O} C$ gives the desired result.
\qed

\subsection*{Proof of Lemma \ref{covariance_bound}}

First, for all $x \in B_R(u^*)$ we have
\begin{equation} \label{bound_values}
\begin{aligned}
\norm{x}_2 &\le \norm{x - u^*}_2 + \norm{u^*}_2 \le R + \norm{u^*}_2 \eqdef m, \\
\norm{\mathcal{G}(x)}_2 &\le \norm{\mathcal{G}(x) - \mathcal{G}(u^*)}_2 + \norm{\mathcal{G}(u^*)}_2 \le C_{\mathcal G} \norm{x - u^*}_2 + \norm{\mathcal{G}(u^*)}_2 \le C_{\mathcal G} R + \norm{\mathcal{G}(u^*)}_2 \eqdef M.
\end{aligned}
\end{equation}
We can also deduce the same bounds for the mean values
\begin{equation} \label{bound_mean}
\norm{\bar{u}}_2 \le \frac{1}{J} \sum_{j=1}^J \norm{u^{(j)}}_2 \le m, \qquad \text{and} \qquad \norm{\bar{\mathcal{G}}}_2 \le \frac{1}{J} \sum_{j=1}^J \norm{\mathcal{G}(u^{(j)})}_2 \le M.
\end{equation}
Then by \eqref{bound_values} and \eqref{bound_mean} we get
\begin{align*}
\norm{C^{up}(u)}_2 = & \; \sup_{x \in \mathbb{R}^L \colon \norm{x}_2 = 1} \left \lVert \frac{1}{J} \sum_{j=1}^J (u^{(j)} - \bar{u}) (\mathcal{G}(u^{(j)}) - \bar{\mathcal{G}})^T x \right \rVert_2 \\
\le & \; \frac{1}{J} \sum_{j=1}^J \left ( \norm{\mathcal{G}(u^{(j)})}_2 + \norm{\bar{\mathcal{G}}}_2 \right ) \left ( \norm{u^{(j)}}_2 + \norm{\bar{u}}_2 \right ) \\
\le & \; 4Mm,
\end{align*}
and defining $C_1 = 4Mm$ we get $(i)$. The argument is similar for the matrix $C^{pp}(u)$, for which we have
\begin{equation*}
\norm{C^{pp}(u)}_2 \le \frac{1}{J} \sum_{j=1}^J \left ( \norm{\mathcal{G}(u^{(j)})}_2 + \norm{\bar{\mathcal{G}}}_2 \right )^2 \le 4M^2,
\end{equation*}
and defining $C_2 = 4M^2$ we get $(ii)$. Before proving $(iii)$ and $(iv)$, we need the following estimates for two ensemble of particles $u_1$ and $u_2$
\begin{equation} \label{diference_uG}
\begin{aligned}
\norm{\bar{u}_1 - \bar{u}_2}_2 &= \left \lVert \frac{1}{J} \sum_{j=1}^J (u_1^{(j)} - u_2^{(j)}) \right \rVert_2 \le \frac{1}{J} \sum_{j=1}^J \norm{u_1^{(j)} - u_2^{(j)}}_2 = \norm{u_1 - u_2}, \\
\norm{\bar{\mathcal{G}}_1 - \bar{\mathcal{G}}_2}_2 &= \left \lVert \frac{1}{J} \sum_{j=1}^J (\mathcal{G}(u_1^{(j)}) - \mathcal{G}(u_2^{(j)})) \right \rVert_2 \le \frac{C_{\mathcal G}}{J} \sum_{j=1}^J \norm{u_1^{(j)} - u_2^{(j)}}_2 = C_{\mathcal G} \norm{u_1 - u_2}.
\end{aligned}
\end{equation}
Then we have
\begin{equation}
\begin{aligned}
& \norm{C^{up}(u_1) - C^{up}(u_2)}_2 \\
&\qquad = \sup_{x \in \mathbb{R}^L \colon \norm{x}_2 = 1} \left \lVert \frac{1}{J} \sum_{j=1}^J \left [ (u_1^{(j)} - \bar{u}_1) (\mathcal{G}(u_1^{(j)}) - \bar{\mathcal{G}}_1)^T x - (u_2^{(j)} - \bar{u}_2) (\mathcal{G}(u_2^{(j)}) - \bar{\mathcal{G}}_2)^T x \right ] \right \rVert_2 \\
&\qquad \le \frac{1}{J} \sum_{j=1}^J \left( \norm{u_1^{(j)}}_2 + \norm{\bar{u}_1}_2 \right) \left( \norm{\mathcal{G}(u_1^{(j)}) - \mathcal{G}(u_2^{(j)})}_2 + \norm{\bar{\mathcal{G}}_2 - \bar{\mathcal{G}}_1}_2 \right) \\
&\qquad\quad + \frac{1}{J} \sum_{j=1}^J \left( \norm{u_1^{(j)} - u_2^{(j)}}_2 + \norm{\bar{u}_2 - \bar{u}_1}_2 \right) \left( \norm{\mathcal{G}(u_2^{(j)})}_2 + \norm{\bar{\mathcal{G}}_2}_2 \right),
\end{aligned}
\end{equation}
and since $\mathcal G$ is Lipschitz and due to \eqref{bound_values}, \eqref{bound_mean}, \eqref{diference_uG}, we obtain
\begin{align*}
\norm{C^{up}(u_1) - C^{up}(u_2)}_2 &\le 2m ( C_{\mathcal G} J \norm{u_1 - u_2} + C_{\mathcal G} \norm{u_1 - u_2} ) + ( J \norm{u_1 - u_2} + \norm{u_1 - u_2} ) 2M \\
&\le 2 (J+1) (mC_{\mathcal G} + M) \norm{u_1 - u_2},
\end{align*}
and defining $C_3 = 2 (J+1) (mC_{\mathcal G} + M)$ we get $(iii)$. The argument is similar for the matrix $C^{pp}(u)$, for which we have
\begin{equation}
\begin{aligned}
\norm{C^{pp}(u_1) - C^{pp}(u_2)}_2 &\le \frac{1}{J} \sum_{j=1}^J \left( \norm{\mathcal{G}(u_1^{(j)})} + \norm{\bar{\mathcal{G}}_1}_2 \right) \left( \norm{\mathcal{G}(u_1^{(j)}) - \mathcal{G}(u_2^{(j)})}_2 + \norm{\bar{\mathcal{G}}_2 - \bar{\mathcal{G}}_1}_2 \right) \\
&\quad + \frac{1}{J} \sum_{j=1}^J \left( \norm{\mathcal{G}(u_1^{(j)}) - \mathcal{G}(u_2^{(j)})}_2 + \norm{\bar{\mathcal{G}}_2 - \bar{\mathcal{G}}_1}_2 \right) \left( \norm{\mathcal{G}(u_2^{(j)})} + \norm{\bar{\mathcal{G}}_2} \right) \\
&\le 4(J+1)M C_{\mathcal G},
\end{aligned}
\end{equation}
and defining  $C_4 = 4(J+1)MC_{\mathcal G}$ we get $(iv)$, which concludes the proof.
\qed

\subsection*{Proof of Lemma \ref{fh_goes_to_0}}

Let us consider an ensemble $u \in \mathcal U_{J,M}$ with particles $u^{(j)} \in \R^M$, for $j = 1, \ldots, J$. For each particle we have 
\begin{equation}
\norm{\mathcal{G}^0_h(u^{(j)}) - \mathcal{G}^0(u^{(j)})}_2 = \norm{\mathcal{O}(\mathcal{S}^0_h(u^{(j)})) - \mathcal{O}(\mathcal{S}^0(u^{(j)}))}_2 
\le C_{\mathcal O} \norm{p^0_h(u^{(j)}) - p^0(u^{(j)})}_{L^2(\Omega)},
\end{equation}
where we write explicitly the dependence of the solutions $p^0$ and $p^0_h$ on the particle they are generated by. Then due to standard a priori error estimates of FEM (see e.g. \cite[Theorem 3.2.5]{Cia02}) and higher order boundary regularity results for elliptic partial differential equations (see e.g. \cite[Theorem 6.3.5]{Eva10}) we have for all $j = 1, \ldots, J$
\[ \norm{p^0_h(u^{(j)}) - p^0(u^{(j)})}_{L^2(\Omega)} \le C \abs{p^0(u^{(j)})}_{H^{s+1}(\Omega)} h^{s+1} \le C \norm{f}_{H^{q-1}(\Omega)} h^{s+1}, \]
where $C>0$ is a constant independent of $h$.
Therefore, we obtain
\begin{equation}
\tilde{e}(h, u) = \frac1J \sum_{j=1}^J \norm{\mathcal{G}^0_h(u^{(j)}) - \mathcal{G}^0(u^{(j)})}_2 \le C_{\mathcal O} C \norm{f}_{H^{q-1}(\Omega)} h^{s+1},
\end{equation}
and defining $\tilde{K} = C_{\mathcal O} C \norm{f}_{H^{q-1}(\Omega)}$ gives the desired result.
\qed

\subsection*{Proof of Lemma \ref{equivalence_convergence_DW1}}

We follow the same steps of the proof of Theorem 5.9 in \cite{San15}. Let us first recall the duality formula for the Wasserstein distance with $p=1$
\[ W_{1,s}(\mu_n, \mu) = \sup_{\varphi \in \Phi} \left \{ \int_{B_R(u^*)} \varphi d(\mu_n - \mu) \right \}, \]
where $\Phi$ is the set of all globally Lipschitz continuous functions $\varphi \colon B_R(u^*) \to \R$ with Lipschitz constant $C_{\mathrm{Lip}} \le 1$. Note that if $\varphi \in \Phi$, then also $- \varphi \in \Phi$. Hence we deduce that
\begin{equation}
W_{1,s}(\mu_n, \mu) = \sup_{\varphi \in \Phi} \left \{ \left | \int_{B_R(u^*)} \varphi d(\mu_n - \mu) \right | \right \}.
\end{equation}
Then we have
\begin{equation}
\sup_{\varphi \in \Phi} \mathbb{E}_{\xi} \left [ \left | \int_{B_R(u^*)} \varphi d(\mu_n - \mu) \right | \right ] \le \mathbb{E}_{\xi} \left [ \sup_{\varphi \in \Phi} \left \{ \left | \int_{B_R(u^*)} \varphi d(\mu_n - \mu) \right | \right \} \right ] = \mathbb{E}_{\xi} [ W_{1,s}(\mu_n, \mu) ],
\end{equation}
where the right hand side vanishes by hypothesis. Therefore we obtain
\begin{equation} \label{for_all_phi}
\mathbb{E}_{\xi} \left [ \left | \int_{B_R(u^*)} \varphi d \mu_n - \int_{B_R(u^*)} \varphi d \mu \right | \right ] \to 0,
\end{equation}
for all $\phi \in \Phi$. Finally, we extend \eqref{for_all_phi} to all Lipschitz functions by linearity and to all bounded continuous functions by density, thus proving the desired result.
\qed

\def\cprime{$'$}

\end{document}